\documentclass[10pt]{amsart}
\usepackage{amsmath,amsfonts,amssymb,amsthm}
\usepackage{graphics,graphicx}
\usepackage[colorlinks,hyperindex,linkcolor=blue,
citecolor=purple]{hyperref}
\hypersetup{
pdfauthor = {Francis Bischoff, Alvaro del Pino Gomez, Aldo Witte},
 pdftitle= {Jets of foliations and bk algebroids},
 pdfsubject = {Lie Theory, singular foliations, log-Geometry, Lie algebroids}}
\usepackage{tikz-cd}
\usepackage{tikz}
\usepackage[all]{xy}
\usepackage{subfiles}
\usetikzlibrary{matrix,arrows,decorations.pathmorphing}

\newcommand{\ff}{\mathcal{F}}

\newcommand{\gfrac}{\mathfrak{g}}

\newcommand{\rr}{\mathbb{R}}

\DeclareMathOperator{\at}{at}

\usepackage{thmtools}
\declaretheorem[style=definition,qed=$\diamondsuit$]{definition}
\declaretheorem[style=definition,qed=$\triangle$,sibling=definition]{example}
\declaretheorem[style=plain,sibling=definition]{theorem}
\declaretheorem[style=plain,sibling=definition]{lemma}
\declaretheorem[style=plain,sibling=definition]{proposition}
\declaretheorem[style=plain,sibling=definition]{corollary}
\declaretheorem[style=definition,qed=$\diamondsuit$,sibling=example]{claim}
\declaretheorem[style=definition,sibling=example]{question}
\declaretheorem[style=definition,qed=$\diamondsuit$,sibling=claim]{remark}

\newtheorem*{claim*}{Claim}
\numberwithin{theoremalpha}{section}

\numberwithin{equation}{section}
\numberwithin{definition}{section}
\numberwithin{theorem}{section}
\numberwithin{proposition}{section}
\numberwithin{lemma}{section}
\numberwithin{example}{section}
\numberwithin{remark}{section}
\numberwithin{corollary}{section}
\numberwithin{question}{section}
\numberwithin{problem}{section}
\numberwithin{assumption}{section}

\newtheorem*{theorem*}{Theorem}

\newtheoremstyle{named}{}{}{\itshape}{}{\bfseries}{.}{.5em}{\thmnote{#3}#1}
\theoremstyle{named}

\numberwithin{equation}{section}
\address{University of Regina, 3737 Wascana Parkway, Regina, Saskatchewan, S4S 0A2, Canada}
\email{Francis.Bischoff@uregina.ca}

\address{Department of Mathematics, Bundesstraße 55, 20146 Hamburg, Germany}
\email{aldowitte@hotmail.nl}

\title[Polynomial degeneration and the Poisson geometry of truncated polynomials]{Polynomial degeneration and the \\ Poisson geometry of truncated polynomials}
\author{Francis Bischoff and Aldo Witte}
\begin{document}

\begin{abstract}
We develop a formalism for studying geometric structures that degenerate to
polynomial order along a hypersurface $W \subset M$. We then demonstrate it in the
study of Poisson geometry, where it leads to methods for constructing
generically symplectic Poisson structures with non-trivial symplectic variation
along their degeneracy locus. This is in contrast to $b/\log$-symplectic and $b^k$-symplectic
structures, where this variation always vanishes. Our main
insight is that the higher residue data along the hypersurface is controlled by
a group of transverse diffeomorphisms, which in our case is the group $G_k$ of
degree-$k$ truncated polynomials. We show that the symplectic variation of our
Poisson structures is determined by the obstruction to lifting a
$G_k$-representation of the fundamental group $\pi_1(W)$ to $G_{k+1}$, and we
construct maps from a $G_k$-character variety into the moduli space of Poisson
structures, with the variation detecting the non-triviality of the resulting
families.
\end{abstract}


\maketitle

\tableofcontents

\section{Introduction}
Geometric structures degenerating along a hypersurface $W \subset M$ arise constantly throughout differential and algebraic geometry. This happens, for example, when a flat connection develops a singularity, or when a Poisson structure drops rank. Often, there is rich information, such as residues, that is hidden along the hypersurface and which is crucial to properly understanding the degeneration. A prominent example of this phenomenon occurs in the work of Melrose \cite{MR1348401}, who extended the Atiyah-Singer index theorem to manifolds with boundary. In this case, the operators he was interested in, such as the Laplacian, degenerate along the boundary, and in order to properly handle these, Melrose developed the b-calculus. 

In this paper, we develop a formalism for treating \emph{polynomial} degenerations and demonstrate it in the study of Poisson geometry, where it leads to methods for constructing generically symplectic Poisson structures with interesting new properties. In this case, the hidden data along the hypersurface takes the form of a sequence of `higher residues'. This greatly extends the b-calculus, which involves only `linear' degenerations, as well as the work of Scott \cite{MR3523250} which treats the simplest type of polynomial degeneration. 

Our formalism is based on a study of differential forms with prescribed singularity-type along the hypersurface $W$, and there are two key insights which inform our approach. First, the main objects controlling the way in which a geometric structure can degenerate along $W$ are groups of transverse diffeomorphisms. In our case, these are the groups
\[
G_{k} = \{ a_0 z + a_1 z^2 + ... + a_{k-1}z^{k} \ | \ a_{i} \in \mathbb{R}, \ a_0 \neq 0 \}
\]
of degree $k$-truncated polynomials with product given by function composition, and we think of them as the diffeomorphism groups of $k$-th order fat points $\mathsf{Spec}(\mathbb{R}[z]/(z^{k+1}))$. Their relevance comes from the fact that a degenerating geometric structure must be controlled along the $k$-th order neighbourhood of $W$, which is locally isomorphic to 
\[
W \times \mathsf{Spec}(\mathbb{R}[z]/(z^{k+1})).
\] 

In order to specify what it means for a geometric structure to degenerate along $W$ to order $k+1$, we must choose a foliation on the $k$-th order neighbourhood. Since these foliations are classified by their holonomy representations \cite{foliationpaper, francis2024singular}
\[
\phi: \pi_1(W) \to G_{k},
\]
the character variety $\mathsf{Hom}(\pi_1(W), G_k)/G_k$ provides moduli for $k+1$-st order degenerations. 

Our second main insight is that the higher residues of a degenerating Poisson structure are tightly controlled by the $G_{k}$-representation. As a result, the $G_k$-character variety plays a central role in the study of generically symplectic Poisson structures. This observation leads to our two main results, in which the geometry of $G_k$-representations governs the Poisson geometry: 
\begin{enumerate}
\item The symplectic leaves of a generically symplectic Poisson structure degenerating along $W$ can vary non-trivially. In fact, the symplectic variation is prescribed by an \emph{extension class} which measures the failure of a $G_{k}$-representation to lift to $G_{k+1}$: 
\[
\begin{tikzcd}
& G_{k+1} \arrow[d] \\
\pi_1(W) \arrow[r, "\phi"] \arrow[ur, dashed] & G_k
\end{tikzcd}
\]
This is in contrast to the case of $b/\log$ and $b^m$-symplectic structures, where the symplectic variation is always zero \cite{GUILLEMIN2014864, MR3523250}. 
\item The close relationship between the higher residues and the holonomy representation $\phi$ allows us to move around in the space of Poisson structures by varying $\phi$. More precisely, we show that deformations of the representation can often be lifted uniquely to deformations of the Poisson structure. This leads to the construction of maps from the $G_k$-character varieties into the moduli space of Poisson structures, providing a method for constructing large families of generically symplectic Poisson structures with non-vanishing symplectic variation. The variation then serves as a detecting invariant for the non-triviality of these families, analogous to the Torelli theorem in algebraic geometry.
\end{enumerate}

\subsubsection*{Outline of paper} We now give a brief outline of the main steps of the paper. The formalism we employ to describe degenerations is that of \emph{hypersurface (HS) algebroids}, a family of Lie algebroids introduced in \cite{foliationpaper}. In Section \ref{sec:HSalg}, we recall the main features of their theory, including their Riemann-Hilbert classification and the extension class. In Section \ref{sec:Examples} we introduce the two running examples of the paper: an HS algebroid on a mapping torus associated to the Arnold cat map, and a family of universal hypersurface algebroids. In Section \ref{sec:algcoh}, we develop the de Rham theory of HS algebroids. This provides a formalism for describing the global higher residues of degenerating geometric structures and leads to a generalization of the cohomology results  for the $b$-tangent bundle of Mazzeo-Melrose \cite[Proposition 2.49]{MR1348401} and the $b^k$-tangent bundles of Scott \cite[Proposition 4.3]{MR3523250}. We then apply this formalism in Section \ref{sec:symalg} to the study of generically symplectic Poisson structures. Here, we introduce HS symplectic structures as symplectic forms on HS algebroids and prove our first main result (Theorem \ref{prop:structureonD}), relating the symplectic variation to the extension class. In Section \ref{sec:defconst}, we prove our second main result, which describes how to lift deformations of HS algebroids to deformations of HS symplectic structures. Finally, this is applied in Section \ref{sec:deformalongcharvar} to construct maps from the $G_{k}$-character variety into the moduli space of Poisson structures. This is illustrated in the case of mapping tori, where we construct arbitrarily large families of HS symplectic structures with non-vanishing symplectic variation. We end the paper with Section \ref{sec:ques}, which raises some open questions. 


\subsubsection*{Relationship to the existing literature}
This paper arose by splitting \cite{bischoff2023jets} into two expanded papers, the first being \cite{foliationpaper}. For this reason, some of the results in this paper have appeared in the arXiv preprint \cite{bischoff2023jets}. 

The Poisson structures in this paper fit into the framework of symplectic Lie algebroids. The study of this class of Poisson structures was initiated by Nest, Tsygan \cite{NestTsygan2001} and Radko \cite{Radko2001ACO}, and further studied in a number of papers, including \cite{GUILLEMIN2014864, Gualtieri-Li-2012, gualtTrop, matviichuk2020local, MR4229238, MR4523256, MR4257086, MR4236806, MR3952555}. The Lie algebroids involved have often been the $b$-tangent bundle \cite{MR1348401}, also known as the logarithmic tangent bundle \cite{deligne1971hodge,MR0417174,saito1980theory}, and the $b^{k}$-tangent bundles of Scott \cite{MR3523250}. Our paper makes use of the more general hypersurface algebroids. 

\subsection*{Acknowledgements}

We are grateful to Ioan Marcut, Marius Crainic, and Marco Zambon for several helpful discussions, and to Marco Gualtieri for providing valuable feedback for rewriting the original version of this paper. We are especially thankful to \'Alvaro del Pino for listening to our endless conversations about the project.  The first author is supported by an NSERC Discovery grant. The second author was supported by FWO and FNRS under EoS projects G0H4518N and G012222N, FWO-EoS Project G083118N, directly by the University of Antwerp via BOF 49756 and is supported by the Deutsche Forschungsgemeinschaft (DFG, German Research Foundation) -- SFB 1624 -- ``Higher structures, moduli spaces and integrability'' -- 506632645.

%
%
\section{Review of hypersurface algebroids}\label{sec:HSalg}

In order to describe degenerating geometric structures, we use the formalism of hypersurface (HS) algebroids. This is a family of Lie algebroids that was introduced in \cite{foliationpaper, bischoff2023jets} as a generalization of the $b^{k}$-algebroids of Scott \cite{MR3523250}. In this section, we recall some of their basic properties, which were established in \cite{foliationpaper}. A version of the classification result was obtained independently in \cite{francis2024singular}. 

\begin{definition}
Let $M$ be a manifold and let $W \subset M$ be a closed hypersurface. A \emph{hypersurface (HS) algebroid} for $(W, M)$ is a Lie algebroid $A \Rightarrow M$ with the property that the anchor map $\rho: A \to TM$ is an isomorphism on the complement $M \setminus W$ and has corank $1$ on $W$. The algebroid $A$ has \emph{order} $k$ if the determinant of the anchor, $\det(\rho) : \det(A) \to \det(TM)$, vanishes along $W$ to order $k$. We always assume that the order is finite. 
\end{definition}

A fundamental result about HS algebroids is their local normal form, which follows from a straightforward application of the splitting theorem for Lie algebroids \cite{Duf01, Fer02, Wei00, BLM19}. 

\begin{proposition}\cite[Proposition 1.4]{foliationpaper} \label{normalformprop}
Let $A \Rightarrow M$ be an HS algebroid of order $k$ for $(W, M)$. Then there are local coordinates $(z, x_2, ..., x_n)$ for $M$ such that $W = \{ z = 0 \}$, and such that $A$ is spanned by the following sections 
\[
z^k \partial_z, \partial_{x_2}, ..., \partial_{x_n}. \hfill \qedhere
\]
\end{proposition}

When working with HS algebroids it is convenient to choose a tubular neighbourhood embedding for $W$. This is because HS algebroids on the total space of a line bundle admit an explicit description in terms of differential forms satisfying a Maurer-Cartan equation. 

\begin{theorem}\cite[Proposition 3.1 and Theorem 6.4]{foliationpaper} \label{MCequation}
Let $\pi : L \to W$ be a real line bundle. There is a bijection between HS algebroids of order $k+1$ for $(W, L)$ and the data of `splittings' $\sigma = \nabla + \sum_{i = 1}^{k-1} \eta_i$, where $\nabla$ is a flat connection on $L$ and $\eta_{i} \in \Omega^1_{W}(L^{-i})$ are twisted differential forms satisfying the following Maurer-Cartan equation 
\[
d^{\nabla} \eta_{r} + \frac{1}{2} \sum_{i + j = r} (j-i) \eta_{i} \wedge \eta_{j} = 0,
\]
for $r = 1, ..., k-1$, and where $d^{\nabla}$ is the differential on $\Omega^{\bullet}_W(L^{-r})$ induced by $\nabla$. 
\end{theorem}

We can explicitly describe the HS algebroid $A(\sigma)$ associated to a splitting $\sigma = \nabla + \sum_{i = 1}^{k-1} \eta_{i}$. Let $\mathcal{E}$ denote the Euler vector field on $L$. Viewing a section $s \in L^{-i}$ as a function with fibrewise degree $i$ on $L$, we can form a vertical vector field $s \mathcal{E} \in \mathfrak{X}(L)$. Then given $X \in \mathfrak{X}(W)$, a vector field on $W$, the splitting $\sigma$ gives rise to the following vector field on the total space of $L$
\begin{equation} \label{algebroidsections}
\sigma(X) = \nabla_{X}  + \sum_{i = 1}^{k-1} \eta_{i}(X) \mathcal{E}.
\end{equation}
The HS algebroid $A(\sigma)$ associated to $\sigma$ is generated by the vector fields $\sigma(X)$ along with $L^{-k} \mathcal{E}$, which are vertical vector fields whose coefficients have fibrewise degree $k+1$. A local basis of $A(\sigma)$ is given by $\sigma(X_1), ..., \sigma(X_{n}), s \mathcal{E}$, where $X_{1}, ..., X_{n}$ is a local basis of $TW$, and $s \in L^{-k}$ is a non-vanishing section. 

\begin{remark}[Scott's $b^{k+1}$-algebroids]
In \cite{MR3523250}, the $b^{k+1}$-tangent bundles are introduced and defined via the use of semi-global defining functions $f$ for the hypersurface $W$. Note that there is not a unique $b^{k+1}$-tangent bundle, since it depends on the choice of defining function $f$. However, as we show in \cite{foliationpaper}, all choices lead to isotopic algebroids. In terms of Theorem \ref{MCequation}, we obtain a Scott $b^{k+1}$-tangent bundle by taking $L = \mathbb{R} \times W$ to be the trivial bundle, equipped with the trivial connection $\nabla = d$ and all $\eta_i = 0$.  
\end{remark}

\subsection{The truncated polynomial group} \label{sec:structuregroup}
We recall a few basic properties of $G_{k}$, the group of $k$-truncated polynomials. Its elements have the form 
\begin{equation} \label{polynomial}
f(z) = a_0 z + a_1 z^2 + ... + a_{k-1} z^k,
\end{equation}
where $a_{i} \in \mathbb{R}$ and $a_{0} \neq 0$, and the group structure is given by function composition followed by truncation of any terms of degree above $k$. By forgetting the highest order term in $f(x)$, we obtain a smooth homomorphism $p_{k} : G_{k} \to G_{k-1}$ whose kernel is isomorphic to the additive group $\mathbb{R}$. Iterating, we obtain a tower of homomorphisms terminating at $G_{1} = \mathbb{R}^*$
\[
... \to G_{k} \to G_{k-1} \to ... \to G_{2} \to G_{1} = \mathbb{R}^*. 
\]
Consider the composition $\pi_{k} : G_{k} \to G_{1}$. This is the map which extracts the linear part $a_{0}$ of a polynomial $f(z)$, or equivalently, takes the derivative $df_{0}$. The kernel $K_{k} = \mathrm{ker}(\pi_{k})$ is the nilpotent subgroup consisting of polynomials where $a_{0} = 1$. By viewing an element of $G_{1} = \mathbb{R}^*$ as a linear diffeomorphism, we obtain a group homomorphism $j_{k} : G_{1} \to G_{k}$ splitting the map $\pi_{k}$. As a result, we have the following semi-direct product decomposition 
\[
G_{k} \cong K_{k} \rtimes \mathbb{R}^*. 
\]

\subsection{Riemann-Hilbert classification} \label{classificationreminder}
In \cite{foliationpaper} we prove a Riemann-Hilbert correspondence for hypersurface algebroids. This classifies HS algebroids for $(W,M)$ in terms of $G_{k}$-representations of the fundamental group $\pi_1(W)$. We briefly recall the relevant aspects of this classification. 

Choose a basepoint $x \in W$ and consider the space $\mathrm{Hom}(\pi_{1}(W, x), G_{k})$ of homomorphisms, which is an algebraic variety when $\pi_1(W,x)$ is finitely presented, and it is equipped with a conjugation action of $G_{k}$. We will restrict our attention to the action of the connected component of the identity $G_{k}^{0} \subset G_{k}$. 

The projection $\pi_k: G_{k} \to \mathbb{R}^* \cong \mathbb{R} \times \mathbb{Z}/2 \to  \mathbb{Z}/2$ induces a homomorphism 
\[
F: \mathrm{Hom}(\pi_{1}(W, x), G_{k}) \to \mathrm{Hom}(\pi_{1}(W,x), \mathbb{Z}/2) \cong H^{1}(W, \mathbb{Z}/2). 
\]
The latter space classifies real line bundles over $W$ through their first Whitney class. Hence, given a line bundle $L$, with class $w_{1}(L) \in H^{1}(W, \mathbb{Z}/2)$, we may restrict attention to the fibre $\mathrm{Hom}_{L}(\pi_{1}(W, x), G_{k}) := F^{-1}(w_{1}(L))$. This subspace is preserved by the $G_{k}^{0}$-action.

Now given a smooth manifold $M$ containing $W$ as a closed hypersurface, let $\nu_{W} = TM|_{W}/TW$ be the normal bundle, and define 
\[
M_{k}(M, W) = \mathrm{Hom}_{\nu_{W}}(\pi_{1}(W, x), G_{k})/G_{k}^{0}
\]
to be the space of orbits with the quotient topology. We call this space the \emph{$G_{k}$-character variety of $(W,M)$}. 

\begin{theorem}\cite[Corollary 8.42]{foliationpaper} \label{RHthm}
Let $M$ be a manifold and let $W \subset M$ be a closed hypersurface. There is a homeomorphism between the space of isotopy classes\footnote{The isotopies are required to restrict to the identity on $W$.} of hypersurface algebroids of order $(k+1)$ for $(W,M)$ and the $G_{k}$-character variety $M_{k}(M, W)$. Under this correspondence, Scott's $b^{k+1}$-algebroids correspond to the trivial representation. 
\end{theorem}

\subsection{Extension class} \label{extsection} Finally, we recall the fundamental invariant of a hypersurface algebroid $A$ known as the \emph{extension class}. If $A$ has degree $k+1$, then it is a cohomology class $e(A) \in H^2(W, \nu_{W}^{-k})$, where $\nu_{W}$ is the normal bundle of $W$, which is equipped with a flat connection induced by $A$. In \cite[Definition 4.6]{foliationpaper}, $e(A)$ is defined to be the extension class of the restriction $A|_{W}$, which sits in the following short exact sequence of Lie algebroids
\[
0 \to \nu_{W}^{-k} \to A|_{W} \to TW \to 0. 
\]
However, in \cite{foliationpaper} we show that the extension class of $A$ may also be interpreted as the obstruction to \emph{lifting} its holonomy representation $\phi: \pi_{1}(W,x) \to G_{k}$, provided by Theorem \ref{RHthm}, to $G_{k+1}$:
\[
\begin{tikzcd}
& G_{k+1} \arrow[d] \\
\pi_1(W) \arrow[r, "\phi"] \arrow[ur, dashed] & G_k
\end{tikzcd}
\]
Furthermore, if $A(\sigma)$ is an HS algebroid over a line bundle $L \to W$, with splitting $\sigma = \nabla + \sum_{i = 1}^{k -1} \eta_{i}$, then \cite[Corollary 6.8]{foliationpaper} gives the following explicit de Rham form representing the extension class:
\[
e(A) = \frac{1}{2} \sum_{i + j = k} (j-i) \eta_{i} \wedge \eta_{j} = \sum_{i< \frac{k}{2}} (k-2i) \eta_i \wedge \eta_{k-i} \in \Omega^2_W(L^{-k}). 
\]

\section{Examples} \label{sec:Examples}
In this section, we introduce the two running examples of the paper. The first example, taken from \cite{foliationpaper}, is an HS algebroid for which $W$ is a compact mapping torus. The second consists of two closely related families of \emph{universal} hypersurface algebroids which are associated to the groups $G_{k}$ and their nilpotent subgroups $K_{k}$. As the paper progresses, we will further develop these examples and use them to illustrate our main theorems and constructions. 

\subsection{Example: Arnold's cat map}\label{cat} 
Let $\Gamma = \begin{pmatrix} 2 & 1 \\ 1 & 1 \end{pmatrix}$, a matrix with eigenvalues $r^{\pm1} = 1 \pm \varphi^{\pm1}$, where $\varphi = \frac{1 + \sqrt{5}}{2}$ is the golden ratio, and let $\lambda = \log(r)$, which is a non-zero constant. Let $S = \Gamma^{-1} \oplus \Gamma^{-2} \in \mathrm{SL}(\mathbb{Z}^4)$, which induces a diffeomorphism $\phi: \mathbb{T}^4 \to \mathbb{T}^4$ of the $4$-torus. (Recall that $\Gamma$ induces the famous Arnold cat map of the $2$-torus). The pushforward in cohomology $\phi_* : H^{1}(\mathbb{T}^4) \to H^{1}(\mathbb{T}^4)$ induces the decomposition into eigenspaces 
\[
H^{1}(\mathbb{T}^4) = E_{r} \oplus E_{r^{-1}} \oplus E_{r^2} \oplus E_{r^{-2}}, 
\]
where $r^{\pm 1}$ and $r^{\pm 2}$ are the eigenvalues. Let $a_{1}, b_{1}, a_{2}, b_{2}$ be `constant' closed $1$-forms on the torus giving rise to a basis respecting the above eigenspace decomposition. In terms of the standard basis $x_{1}, y_{1}, x_{2}, y_{2}$, we have 
\begin{align*}
a_{1} &= \varphi x_{1} + y_{1}, \qquad b_{1} = -x_1 + \varphi y_1, \\ 
a_{2} &= \varphi x_{2} + y_{2}, \qquad b_{2} = - x_{2} + \varphi y_{2}. 
\end{align*}
Let $W = (\mathbb{T}^4 \times \mathbb{R})/\mathbb{Z}$ be the mapping torus for $\phi$. In other words, it is the quotient space with respect to the $\mathbb{Z}$-action $k \ast (n, \theta) = (\phi^{-k}(n), \theta + k)$. There is a surjective submersion $f: W \to S^1$. Let $\alpha = f^{*}(d\theta)$, where $\theta$ is the coordinate on $S^1$. Using results on the cohomology of mapping tori from Appendix \ref{app:cohomappingtori}, we will construct an order $4$ hypersurface algebroid for $(W, L = W \times \mathbb{R})$. 

Let $\nabla = d - \lambda \alpha$ on $L$. This induces connections $\nabla^{-1} = d + \lambda \alpha$ and $\nabla^{-2} = d + 2 \lambda \alpha$ on $L^{-1}$ and $L^{-2}$ respectively. By Theorem \ref{lambdacohomology} 
\[
H^{1}_{W}(L^{-1}) = H_{W}^{1}(\lambda) \cong E_{r}, \qquad H^{1}_{W}(L^{-2}) = H_{W}^{1}(2\lambda) \cong E_{r^2}, \qquad H^{2}_{W}(L^{-3}) = H_{W}^{2}(3 \lambda) \cong E_{r} \otimes E_{r^2}.
\]
Hence, by Theorem \ref{MCequation}, we obtain an order $4$ HS algebroid $A \Rightarrow L$ by fixing the splitting $\sigma = \nabla + \eta_{1} + \eta_{2}$, where $\eta_{1} = e^{-\lambda \theta} a_1$ and $\eta_{2} = e^{-2 \lambda \theta} a_2$. The extension class is given by the non-trivial element 
\[
e(A) = e^{-3 \lambda \theta} a_{1} \wedge a_{2} \in H^{2}_{W}(L^{-3}).
\] 

\subsection{The truncated Witt algebras} \label{sect:LieAlgcohoE} Let $\mathfrak{g}_{k} = \mathrm{Lie}(G_{k})$ denote the truncated Witt algebra, which is the Lie algebra of $G_{k}$. It has a basis given by the vector fields $z \partial_z, z^2 \partial_z, ..., z^k \partial_z$, whose non-zero Lie brackets are given by 
\[
[z^{i+1} \partial_z, z^{j+1} \partial_z] = (j-i)z^{i + j + 1} \partial_z, \qquad i + j + 1 \leq k.
\] 
Let $\mathfrak{k}_{k} = \mathrm{Lie}(K_{k})$ denote the Lie algebra of $K_{k}$. It contains only the vector fields $z^{i} \partial_z$ with $i \geq 2$. 

In order to discuss the universal hypersurface algebroids, it will be useful to recall a few basic facts about the Chevalley-Eilenberg algebras of these Lie algebras. First, the Chevalley-Eilenberg algebra of $\mathfrak{g}_{k}$ is the cdga given by
\[
\mathrm{CE}(\mathfrak{g}_{k}) = \wedge^{\bullet} \mathfrak{g}_{k}^{*} \cong \mathbb{R}[x_{0}, ..., x_{k-1}].
\]
In the last expression, $x_{i}$ are basis elements of $\mathfrak{g}_{k}^{*}$ dual to $z^{i+1} \partial_{z}$. Hence they have degree $1$ and anti-commute. The differential can be computed as follows 
\[
dx_{r} = \sum_{i = 1}^r i x_{i} \wedge x_{r-i}. 
\]
Since $\mathfrak{k}_{k}$ is a subalgebra of $\mathfrak{g}_{k}$, its Chevalley-Eilenberg algebra is a quotient of $\mathrm{CE}(\mathfrak{g}_{k})$ given by setting $x_{0} = 0$. Both $\mathrm{CE}(\mathfrak{g}_{k})$ and $\mathrm{CE}(\mathfrak{k}_{k})$ are weight-graded by setting $|x_{i}| = -i$, and this grading is respected by the differential $d$ and the product. 

The cohomology of $\mathfrak{k}_{k}$ is closely related to the \emph{twisted} cohomology of $\mathfrak{g}_{k}$. Indeed, let $M_{\lambda} = \mathbb{R}$ be the $\mathfrak{g}_{k}$-module defined by letting $z^{i+1}\partial_z$ act trivially for $i > 0$, and letting $z \partial_{z}$ act as multiplication by $\lambda$. Then $\mathrm{CE}(\mathfrak{g}_{k}, M_\lambda)= \mathrm{CE}(\mathfrak{g}_{k}) \otimes M_{\lambda}$ is the complex computing Lie algebra cohomology valued in $M_{\lambda}$. 

\begin{lemma} \label{CEsubcomplexcalc}
There is a quasi-isomorphism $\mathrm{CE}(\mathfrak{g}_{k}, U_{\lambda}) \cong \mathbb{R}[x_{0}] \otimes \mathrm{CE}(\mathfrak{k}_{k})_{-\lambda}$, where $ \mathrm{CE}(\mathfrak{k}_{k})_{-\lambda}$ denotes the weight $-\lambda$ subcomplex of $\mathrm{CE}(\mathfrak{k}_{k})$. 
\end{lemma}
Differentiating the morphism $p_{k+1} : G_{k+1} \to G_{k}$ defines an exact sequence of Lie algebras 
\begin{equation} \label{eq:algebraseq}
0 \to M_{k} \to \mathfrak{g}_{k+1} \to \mathfrak{g}_{k} \to 0,
\end{equation}
and this has extension class 
\[
e_k = \sum_{i = 1}^{k-1} i x_{k-i} \wedge x_{i} \in H^2(\gfrac_k,M_k). 
\]
Under the quasi-isomorphism of Lemma \ref{CEsubcomplexcalc}, this is sent to a weight $-k$ cohomology class $k_{k} \in H^2(\mathfrak{k}_{k})_{-k}$. This corresponds to the central extension of $\mathfrak{k}_{k}$ given by differentiating the restriction $p_{k+1} : K_{k+1} \to K_{k}$
\begin{equation} \label{eq:algebraseq2}
0 \to \mathbb{R} \to \mathfrak{k}_{k+1} \to \mathfrak{k}_{k} \to 0.
\end{equation}
The cohomology class $k_{k}$ is non-vanishing if and only if $k > 2$.

\subsection{The universal hypersurface algebroids } \label{sec:universalalg}
We now construct a family of universal HS Lie algebroids $U_k$ associated to the groups $K_{k}$. Recall that $K_{k}$ consists of truncated polynomials of the form 
\[
f(z) = z + a_1 z^2 + ... + a_{k-1} z^k.
\]
The coefficients $(a_{1}, ..., a_{k-1})$ give coordinates on the group, identifying it as a smooth manifold with $\mathbb{R}^{k-1}$. Let $K_{k}(\mathbb{Z}) \subset K_{k}$ denote the subgroup of polynomials with integer coefficients. The quotient $X_{k} = K_{k}/K_{k}(\mathbb{Z})$ is a compact manifold. In fact, since $K_{k}$ is contractible, $X_{k}$ is a model for the classifying space $BK_{k}(\mathbb{Z})$. It follows that $\pi_1(X_{k}) \cong K_{k}(\mathbb{Z})$. To see this identification explicitly, note that the fundamental groupoid of $X_{k}$ is given by $\Pi_1(X_{k}) = \mathrm{Pair}(K_{k})/K_{k}(\mathbb{Z}) = (K_{k} \times K_{k})/K_{k}(\mathbb{Z})$, where $K_{k}(\mathbb{Z})$ acts by right multiplication on both factors. There is a well-defined groupoid homomorphism 
\[
H : \Pi_1(X_{k}) \to K_{k}, \qquad (f_1, f_2) \mapsto f_{1} f_{2}^{-1}
\]
given by division. Restricting to the fundamental group based at the coset of $\mathsf{id}$ gives the isomorphism 
\begin{equation}\label{eq:repUk}
h : \pi_1(X_{k}, \mathsf{id}) \to K_{k}(\mathbb{Z}).
\end{equation}
By Theorem \ref{RHthm} this gives rise to a hypersurface algebroid.
\begin{definition}[$K_{k}(\mathbb{Z})$-universal algebroid] 
The \emph{$K_{k}(\mathbb{Z})$-universal algebroid}, denoted $U_k \Rightarrow X_k \times \mathbb{R},$ is the order $k+1$ hypersurface algebroid for $(X_k, X_k \times \mathbb{R})$ associated to the representation \eqref{eq:repUk}.
\end{definition}
Note that the flat connection underlying $U_{k}$ is the trivial one. We can obtain an explicit description of $U_k$ in terms of a splitting by differentiating the homormorphism $H$. Doing so, we find that 
\[
{\rm Lie}(H) = \theta^{R} \in \Omega_{X_{k}}^{\bullet} \otimes \mathfrak{k}_{k},
\] 
the right-invariant Maurer-Cartan form for $K_{k}$, descended to $X_{k}$. The Chevalley-Eilenberg complex of $\mathfrak{k}_{k}$ may be viewed as the differential subalgebra of the de Rham complex $\Omega^{\bullet}_{K_{k}}$ consisting of right-invariant differential forms: $\mathrm{CE}(\mathfrak{k}_{k}) \cong (\Omega^{\bullet}_{K_{k}})^{R}$. Since $X_{k}$ is defined by quotienting out the right action of $K_{k}(\mathbb{Z})$, the right invariant forms descend to $X_{k}$ and we obtain an injective map of cdgas 
\[
\mathrm{CE}(\mathfrak{k}_{k}) = \mathbb{R}[x_{1}, ..., x_{k-1}] \to \Omega^{\bullet}_{X_{k}}. 
\]
By Nomizu's theorem \cite{nomizu1954cohomology}, this is a quasi-isomorphism. Expanding $\theta^R$ in terms of the dual bases for $\mathfrak{k}_{k}$ described in Section \ref{sect:LieAlgcohoE}, we find 
\[
\theta^{R} = \sum_{i = 1}^{k-1} x_{i} \otimes z^{i+1} \partial_{z}.
\] 
Therefore, in terms of the description of Theorem \ref{MCequation}, $U_{k}$ is defined by the splitting $\sigma = d + \sum_{i = 1}^{k-1} \eta_{i}$, where $\eta_{i} = x_{i}$. Let $V_{r} \in \mathfrak{X}(X_{k})^R$ be the right invariant vector field on $X_{k}$ corresponding to $z^{r+1}\partial_{z}$. We conclude:
\begin{proposition}
The universal algebroid $U_k \Rightarrow X_k\times \rr$ is determined by the splitting $\sigma = d + \sum_{i=1}^{k-1}x_i$. It has a global basis of sections given by
\[
V_{1} + t^2 \partial_t, \qquad V_{2} + t^3 \partial_{t}, \qquad ... \qquad  V_{k-1} + t^{k} \partial_{t}, \qquad t^{k+1} \partial_{t},
\]
where $t$ is the linear coordinate on $\mathbb{R}$. Moreover, its extension class is given by $e(U_k) = k_k \in H^2(X_{k}) \cong H^2(\mathfrak{k}_{k})$, the extension class of \eqref{eq:algebraseq2}. Therefore, $e(U_{k})$ is non-vanishing if and only if $k>2$. 
\end{proposition}

\begin{example}[Heisenberg]
When $k = 4$, $K_{4}$ is isomorphic to the Heisenberg group
\begin{equation*}
H(\mathbb{R}) = \lbrace \begin{pmatrix} 1 & u & w \\ 0 & 1 & v \\ 0 & 0 & 1 \end{pmatrix} \ | \ u, v, w \in \mathbb{R} \rbrace.
\end{equation*}
with the isomorphism given by 
\[
\varphi: H(\mathbb{R}) \to K_4, \qquad (u,v,w) \mapsto z + uz^2 + (v+u^2)z^3 + (u^3 + 3uv - w)z^4.
\]
As a result, $X_{4}$ is diffeomorphic to the space $H(\mathbb{R})/H(\mathbb{Z})$. Using $u,v,w$ as coordinates on $X_{4}$, the components of the splitting are given by 
\[
x_1 = du, \qquad x_2 = dv, \qquad x_3 = -dw + v du. 
\]
Therefore, the algebroid $U_4 \Rightarrow X_4 \times \rr$ has global basis 
\[
\partial_u + v \partial_w + t^2 \partial_t, \qquad \partial_v + t^3 \partial_t, \qquad - \partial_w + t^4 \partial_t, \qquad t^5 \partial_t.
\]
Finally, the extension class is the non-zero class $2x_1 \wedge x_3 = -2 du \wedge dw$.
\end{example}
The following proposition justifies calling $U_{k}$ universal. 
\begin{proposition} \label{universal1}
Let $W \subset M$ be a closed hypersurface and let $A \Rightarrow M$ be a hypersurface algebroid for $(W, M)$ of order $k + 1$, for $k \geq 1$. Let $\rho: \pi_{1}(W, x) \to G_{k}$ be a representation classifying $A$ via Theorem \ref{RHthm}. If $\rho$ is valued in $K_{k}(\mathbb{Z})$, then there exists an open neighbourhood $U \subseteq M$ of $W$ and a smooth map of pairs $F: (U, W) \to (X_{k} \times \mathbb{R}, X_{k})$ such that $A|_{U} = F^{!}U_{k}$, the pullback Lie algebroid. 
\end{proposition}

\subsection{The equivariant universal algebroids} \label{sec:equivuniversalalg}
Next, we describe a family of hypersurface algebroids over $G_{k} \times \mathbb{R}$ which are even more fundamental than the $U_{k}$. We obtain their description by directly generalizing the Maurer-Cartan description of $U_{k}$. 

Let $V_{0}, V_{1}, ..., V_{k-1} \in \mathfrak{X}(G_{k})^{R}$ be the right invariant vector fields on $G_{k}$ associated to the basis elements $z^{i + 1}\partial_z$ of $\mathfrak{g}_{k}$, and let $x_{i} \in \Omega^{1}(G_{k})^{R}$, for $i = 0, ..., k-1$, be the dual right invariant $1$-forms. These induce a morphism of cdgas $\mathrm{CE}(\mathfrak{g}_{k}) \to \Omega^{\bullet}(G_{k})$ which identifies the Chevalley-Eilenberg complex with the subcomplex of right invariant differential forms. 

\begin{definition}[$G_{k}$-equivariant universal algebroid]
The \emph{$G_{k}$-equivariant universal algebroid}, denoted $E_{k} \Rightarrow G_{k} \times \mathbb{R}$, is the order $(k+1)$ hypersurface algebroid for $(G_{k}, G_{k} \times \mathbb{R})$ defined by the splitting $\sigma = d + \sum_{i = 0}^{k-1} x_{i}$. 
\end{definition}
The algebroid $E_{k}$ has a natural global basis of vector fields given by 
\[
V_{0} + t \partial_t, \qquad V_{1} + t^2 \partial_t, \qquad  ... \qquad  V_{k-1} + t^{k} \partial_{t}, \qquad t^{k+1} \partial_{t}.
\]
Since both connected components of $G_{k}$ are contractible, this algebroid is isotopic to a Scott $b^{k+1}$-algebroid. However, $E_k$ carries a non-trivial equivariant structure. Indeed, $G_{k} \times \mathbb{R}$ admits both a right $G_{k}$-action given by $(g, t) \ast h = (gh, t)$, and a left $\mathbb{R}^*$-action given by $r \ast (g,t) = (rg, rt)$. These actions commute and preserve $E_{k}$. 

\begin{lemma} \label{equivactionEk}
The algebroid $E_{k}$ admits commuting actions of $G_{k}$ and $\mathbb{R}^*$ which lift their actions on $G_{k} \times \mathbb{R}$. 
\end{lemma}
\begin{proof}
Invariance under the right action follows by construction. Invariance under the left action follows because the vector field $V_{i} +t^{i+1} \partial_t$ pushes forward under the action of $r \in \mathbb{R}^*$ to $r^{-i}( V_{i} +t^{i+1} \partial_t)$. 
\end{proof}

Note that the flat line bundle underlying $E_{k}$ is the trivial bundle $L = G_{k} \times \mathbb{R}$ equipped with the connection $\nabla = d - x_{0}$. This flat bundle and its tensor powers admit a $G_{k}$-equivariant structure. The extension class of $E_{k}$ is represented by the following twisted differential form 
\[
e(E_{k}) = \sum_{i = 1}^{k-1} i x_{k-i} \wedge x_{i} \in \Omega_{G_{k}}^{2}(L^{-k}). 
\]
This form is $G_{k}$-invariant: $e(E_{k}) \in \Omega_{G_{k}}^{2}(L^{-k})^{G_{k}}$, and the $G_{k}$-invariant subcomplex is identified with the twisted Chevalley-Eilenberg complex: $\Omega_{G_{k}}^{\bullet}(L^{-k})^{G_{k}} \cong \mathrm{CE}(\mathfrak{g}_{k}, M_{k})$. Under this identification, $e(E_{k})$ coincides with the extension class $e_k$ of \eqref{eq:algebraseq}. Additionally, $\mathrm{CE}(\mathfrak{g}_{k}, M_{k})$ was shown in Lemma \ref{CEsubcomplexcalc} to be quasi-isomorphic to $\mathbb{R}[x_{0}] \otimes \mathrm{CE}(\mathfrak{k}_{k})_{-k}$. Under these (quasi)-isomorphisms, the extension class $e(E_{k})$ also coincides with $e(U_{k})$. Hence, although $e(E_{k})$ is trivial in twisted de Rham cohomology, it is non-trivial in the invariant cohomology for $k \geq 3$.

\section{Lie algebroid cohomology}\label{sec:algcoh}
The degeneration of geometric structures along a hypersurface $W \subset M$ is encoded in our formalism as differential forms on an order $k+1$ HS algebroid $A$. To motivate this approach, we use the local normal form of Proposition \ref{normalformprop}. In the local form coordinates $(z, x_2, ..., x_{k})$, a differential $A$-form $\omega \in \Omega_{A}^{r}$ may be decomposed in the following way
\begin{equation} \label{localdecomp}
\omega = \sum_{i = 0}^{k} \frac{dz}{z^{i+1}} \wedge \alpha_{i} + \beta,
\end{equation}
with $\alpha_{i} \in \Omega_{W}^{r-1}$ and $\beta \in \Omega^{r}_M$. We therefore see that $A$-forms model a certain class of differential forms on $M$ with poles along $W$. In this decomposition, the $\alpha_{i}$ are the \emph{`higher residues'} (or `principal part') of $\omega$, whereas $\beta$ is the smooth part of the form. However, since the normal form coordinates are only defined locally and a change of coordinates does not preserve the above decomposition, these higher residues do not have a global meaning. The purpose of the present section is to develop a global analogue of the decomposition \eqref{localdecomp}, which will allow us to make global sense of the higher residues of an $A$-form. This will lead to a generalization of the cohomology results for the $b$-tangent bundle of Mazzeo-Melrose \cite[Proposition 2.49]{MR1348401} and the $b^k$-tangent bundles of Scott \cite[Proposition 4.3]{MR3523250}.

\subsection{Local cohomology} \label{sec:localcoho}
Let $\pi: L \to W$ be a real line bundle and let $\sigma = \nabla + \sum_{i = 1}^{k-1} \eta_{i} : TW \to \at^{k}(L)$ be a splitting. By Theorem \ref{MCequation} this determines a hypersurface algebroid $A(\sigma) \Rightarrow \mathrm{tot}(L)$. In this section, we give a detailed study of the algebroid de Rham complex $\Omega_{A(\sigma)}^{\bullet}$. 

We assume that $k \geq 1$, so that $L$ is equipped with a flat connection $\nabla$. Recall that $\mathcal{E} \in \mathfrak{X}(L)$ denotes the Euler vector field on $L$. Using the connection $\nabla$, we can define a dual $1$-form $\epsilon \in \Omega^1_{L \setminus W}$ by requiring that 
\[
\epsilon(\nabla_{X}) = 0 \text{ for all } X \in \mathfrak{X}(W), \qquad \epsilon(\mathcal{E}) = 1. 
\]
In what follows, we will often identify a differential form $\beta \in \Omega_{W}^{j}$ with its pullback $\pi^*(\beta) \in \Omega_{L}^{j}$. We also view sections $s \in L^{-i}$ as functions on $L \setminus W$ with fibrewise degree $i$. 

\begin{lemma} \label{differentialequation}
Let $s \in L^{-i}$, viewed as a function on $L \setminus W$. Then $ds = \nabla(s) + is\epsilon$, where $\nabla(s) \in \Omega^1_{W}(L^{-i})$. 
\end{lemma}

Given a twisted differential form $\eta \in \Omega^j_{W}(L^{-i})$, there is a freedom to view the $L$-factors either as twisting line bundles, or as functions on $L \setminus W$. We write 
\[
\eta_{r} \in \Omega^j_{W}(L^{-i + r}) \otimes L^{-r}
\]
when we view $\eta$ as the $\pi^*(L)^{-r}$-twisted differential form on $L \setminus W$ with fibrewise degree $i - r$. A similar freedom is also available for sections of $\Omega_L^{\bullet}(\pi^*(L^{-i}))$.

The following result identifies the global analogues of the singular $1$-forms $\frac{dz}{z^{r+1}}$ appearing in \eqref{localdecomp}.
\begin{lemma}
For $0 \leq r \leq k$, the following expression defines a twisted algebroid $1$-form 
\[
\tau_{r} = \epsilon_{r} - \sum_{i = 1}^{r-1} \eta_{i,r} \in \Omega^1_{A(\sigma)} \otimes \pi^*L^{-r}.
\]
Viewed as a twisted differential form on $L$, it has a pole of order $r+1$ along $W$. 
\end{lemma}
\begin{proof}
This follows by showing that $\tau_r$ pairs with the generating sections of $A(\sigma)$ --- namely, the sections $\sigma(X)$ from \eqref{algebroidsections} and the vertical sections $s\mathcal{E}$ for $s \in L^{-k}$ --- to give smooth sections of $\pi^*L^{-r}$ over $L$.
\end{proof}

We note the following useful corollary, which follows immediately from the description of local bases of $A(\sigma)$ in Section \ref{sec:HSalg}. 
\begin{corollary} \label{localbasiscor}
Let $\beta_{1}, ..., \beta_{n}$ be a local basis of $T^*W$ and let $s \in L^{k}$ be a non-vanishing section. Then a local basis of $A(\sigma)^*$ is given by the forms $\beta_{i}$ along with $\tau_{k}(s)$. 
\end{corollary}

\subsubsection*{The principal representation}
A hypersurface algebroid $A(\sigma)$ on a line bundle $L$ induces a canonical flat $TW$-connection $\nabla^{\sigma}$ on the bundle
\begin{equation}\label{eq:principal_bundle}
    P_k(L) = \bigoplus_{i=0}^{k} L^i.
\end{equation}
We call $\nabla^{\sigma}$ the \emph{principal representation} of $TW$ induced by $A(\sigma)$, as it serves as the receptacle for the higher residues, or \emph{principal part}, of an $A$-form. We define $\nabla^{\sigma}$ by describing the following $\Omega^\bullet_W$-dg-module structure on $\Omega^\bullet_W(P_k(L))$. Given a section $u \in L^r$, its differential is given by
\begin{equation}\label{eq:principal_diff}
    d^{\sigma} u = d^\nabla u + \sum_{i=1}^{r-1} (i - r)\eta_i \otimes u,
\end{equation}
where $\eta_i \otimes u \in \Omega^1_W(L^{r-i})$.

The global analogue of the decomposition \eqref{localdecomp} is provided by the following result. 

\begin{proposition}  \label{tauchainmap}
There is an isomorphism of cochain complexes 
\[
\tau: \Omega_{L}^{\bullet} \oplus \Omega_{W}^{\bullet}(P_{k}(L))[-1] \to \Omega_{A(\sigma)}^{\bullet}, \qquad (\beta, \sum_{r = 0}^{k} \phi_{r} ) \mapsto \beta + \sum_{r = 0}^{k} \phi_{r} \wedge \tau_{r},
\]
where $\Omega_{W}^{\bullet}(P_{k}(L))[-1]$ denotes the principal dg-module with degrees shifted by $1$, and where $\phi_{r} \wedge \tau_{r} \in \Omega_{A(\sigma)}^{\bullet}$ is obtained by pairing $L^{r}$ and $L^{-r}$ and wedging the differential form components. In particular, we have a canonical isomorphism 
\[
H^{\bullet}(A(\sigma)) \cong H^{\bullet}(W) \oplus H^{\bullet - 1}(W, P_{k}(L)). 
\]
\end{proposition}
\begin{proof}
To prove injectivity of $\tau$, we look at the most singular term in $\tau(\beta, \sum_{r = 0}^{k} \phi_{r}) = 0$ to show that $\phi_{k} = 0$, and then proceed inductively. To prove surjectivity, we start by working locally on $W$. By Corollary \ref{localbasiscor}, a form $\omega \in \Omega^p_{A(\sigma)}$ can be expanded as 
\[
\omega = \alpha \wedge \tau_{k}(s) + \gamma, 
\]
where $\alpha$ and $\gamma$ are horizontal forms, and $s \in L^{k}$ is non-vanishing. Now Taylor expand $\alpha$ in the vertical direction: 
\[
\alpha = \alpha_0 + \alpha_1 + ... + \alpha_{k} + \tilde{\alpha}_{k+1},
\]
where $\alpha_{i} \in \Omega^{p-1}_{W}(L^{-i})$ and $ \tilde{\alpha}_{k+1} \in \Omega^{p-1}_{L}$ vanishes to order at least $k + 1$ along $W$. Hence $\tilde{\alpha}_{k+1} \wedge \tau_{k}(s)$ is smooth. Furthermore, 
\[
\alpha_{i} \wedge \tau_{k}(s) =  \alpha_{i}s  \wedge \tau_{k-i} + \sum_{j = k-i}^{k-1} s \alpha_{i} \wedge \eta_{j}, 
\]
where the terms in the right-hand sum are all smooth. Hence, we have $\omega = \sum_{r = 0}^{k} \phi_{r} \wedge \tau_{r} + \beta,$ where $\phi_{r} = \alpha_{k-r} s \in  \Omega^{p-1}_{W}\otimes L^{r}$ and $\beta \in \Omega_{L}^{p}$ is the smooth form obtained by adding up $\gamma, \tilde{\alpha}_{k+1} \wedge \tau_{k}(s)$ and all of the terms $\sum_{j = k-i}^{k-1} s \alpha_{i} \wedge \eta_{j}$. This shows that $\tau$ is surjective locally on $W$. From this we get global surjectivity using a partition of unity.

In order to show that $\tau$ is a cochain map, it suffices to check that $d(u \wedge \tau_{r}) = \tau(d^{\sigma} u)$ for $u \in L^{r}$. This is a straightforward computation making use of Equation \ref{eq:principal_diff}, the Maurer-Cartan equation from Corollary \ref{MCequation}, and $d^{\nabla} u = du + r u \epsilon$ from Lemma \ref{differentialequation}.
\end{proof}

\subsubsection*{Restriction to the hypersurface} The algebroid $A(\sigma)$ may be restricted to the submanifold $W$, giving the following short exact sequence 
\begin{equation} \label{restrictedses}
0 \to L^{-k} \to A(\sigma)|_{W} \to TW \to 0. 
\end{equation}
The twisted algebroid form $\tau_{k}$ can be viewed as a retraction $A(\sigma)|_{W}  \to L^{-k}$ of the inclusion $L^{-k} \to A(\sigma)|_{W}$, albeit one that does not respect the Lie algebroid structure. This induces a decomposition 
\begin{equation} \label{taudecomp} 
\Omega_{A(\sigma)|_{W}}^{\bullet} \cong \Omega_{W}^{\bullet} \oplus \Omega_{W}^{\bullet}(L^k)[-1]
\end{equation} 
as graded vector spaces. Now consider the restriction morphism of cdga's $\mathcal{R}: \Omega_{A(\sigma)}^{\bullet} \to \Omega^{\bullet}_{A(\sigma)|_{W}}$. In terms of the decompositions of Proposition \ref{tauchainmap} and Equation \eqref{taudecomp}, this is given by 
\begin{equation} \label{restrictionexp}
\mathcal{R} : \Omega_{L}^{\bullet} \oplus  \Omega_{W}^{\bullet}(P_{k}(L))[-1] \to  \Omega_{W}^{\bullet} \oplus \Omega_{W}^{\bullet}(L^k)[-1], \qquad (\beta, \sum_{r = 0}^{k} \phi_{r}) \mapsto (\beta|_{W} + \sum_{r = 1}^{k-1} \phi_{r} \wedge \eta_{r}, \phi_{k}). 
\end{equation}

\subsubsection*{A universal cdga} Recall the equivariant universal algebroid $E_{k} \Rightarrow G_{k} \times \mathbb{R}$ from Section \ref{sec:equivuniversalalg} and recall that it admits commuting actions of $\mathbb{R}^*$ and $G_{k}$. These actions are inherited by the algebroid de Rham complex $\Omega_{E_{k}}^{\bullet}$, as well as the associated principal dg module. Hence, we may consider the $G_{k}$-invariant subcomplex of $\Omega^\bullet_{G_{k}}(P_k(L))$ and this gives the following dg-module over the Chevalley-Eilenberg algebra $\mathrm{CE}(\mathfrak{g}_{k})$: 
\begin{equation} \label{univcdga}
S_{k} = \oplus_{r = 0}^{k} \mathrm{CE}(\mathfrak{g}_{k})t_r,
\end{equation}
where $t_{r}$ is a generator of degree $1$ corresponding to the constant unit section of $L^r$. By Equation \eqref{eq:principal_diff} the differential is given by 
\[
dt_{r} = \sum_{i = 0}^{r-1}(i-r)x_{i} t_{r-i}. 
\]
The $\mathbb{R}^*$-action shows up as a weight grading on $S_{k}$, extending the weight grading on $\mathrm{CE}(\mathfrak{g}_{k})$, and such that $|t_{i}| = -i$. 
\begin{remark}
Under the isomorphism of Proposition \ref{tauchainmap}, $S_{k}$ is mapped to a \emph{subalgebra} of $\Omega_{E_{k}}^{\bullet}$. Hence, $S_{k}$ is a cdga defined over $\mathrm{CE}(\mathfrak{g}_{k})$. Its multiplication is given by the following curious formulas
\[
t_{0} t_{r} = x_{1} t_{r-1} + x_{2} t_{r-2} + ... + x_{r-1} t_{1}
\] 
and 
\begin{equation} \label{Skprod}
t_{s} t_{r} = x_{s}t_{r} + ... + x_{r-1} t_{s + 1}
\end{equation}
for $0 < s < r$. 
\end{remark}

\subsection{Global cohomology} \label{sec:globalcoho} Let $W \subset M$ be a hypersurface in a manifold, let $L = \nu_{W}$ denote the normal bundle, and let $A \Rightarrow M$ be an order $k +1$ hypersurface algebroid for $(W,M)$. In this section we study the cohomology of $A$ and globalize the results from Section \ref{sec:localcoho}, thereby obtaining a fully global definition of the higher residues of an algebroid form. This will be applied in Section \ref{sec:defconst} when we study deformations. 

By choosing a tubular neighbourhood embedding $\phi: L \to U \subseteq M$ for $W$, we determine a splitting $\sigma$, as well as a cochain map
\[
\mathrm{Princ} : \Omega_{A}^{\bullet} \to \Omega_{W}^{\bullet}(P_k(L))[-1]
\]
which extracts the `principal' part of a form. More precisely, this map is defined by restricting an algebroid form to a neighbourhood of $W$, applying the inverse to the isomorphism of Proposition \ref{tauchainmap}, and then projecting to the principal dg module. There is also an inclusion of the smooth forms $\Omega_M^{\bullet} \to \Omega_{A}^{\bullet}$ into the kernel of $\mathrm{Princ}$ induced by the anchor map $\rho: A \to TM$. 

\begin{theorem} \label{decomp1proof}
There is a short exact sequence of chain complexes 
\begin{equation} \label{exactglobalcomp}
0 \to \Omega_{M}^{\bullet} \to \Omega_{A}^{\bullet} \to \Omega_{W}^{\bullet}(P_{k}(L))[-1] \to 0,
\end{equation}
which induces a short exact sequence of cohomology groups 
\begin{equation} \label{exactglobalcoh}
0 \to H^{\bullet}(M) \to H^{\bullet}(A) \to H^{\bullet - 1}(W, P_{k}(L)) \to 0.
\end{equation}
\end{theorem}
\begin{remark} Let $A$ be a Scott $b^{k+1}$-algebroid, associated to a $k$-jet of defining function $f$ for $W$. Using $f$ to trivialize $\nu_W$, the splitting associated to $A$ is given by $\sigma = d$, where $d$ is the trivial flat connection. Consequently $P_k(L)$ is isomorphic to the direct sum of $(k+1)$-copies of the trivial local system. Therefore $H^{\bullet - 1}(W, P_{k}(L))$ is isomorphic to $H^{\bullet - 1}(W)^{k+1}$, which recovers Proposition 4.3 from \cite{MR3523250}.
\end{remark}
\begin{proof}
The proof consists simply in producing an appropriate splitting of $\mathrm{Princ}$. However, the details of this are slightly technical and require a few steps. 
We start by fixing the following preliminary data: 
\begin{itemize}
\item a smooth bump function $\psi_{U}$ which is supported in $U$ and equal to $1$ on a sub-neighbourhood $U' \subset U$ of $W$. 
\item a metric $h$ on $L$, viewed as a section $h \in L^{-2}$. 
\end{itemize}
First, the bump function $\psi_U$ allows us to construct a splitting of $\mathrm{Princ}$ as follows
\begin{equation} \label{ogsplitting}
s: \Omega_{W}^{\bullet}(P_{k}(L))[-1] \to \Omega^{\bullet}_{A}, \qquad \phi \mapsto \psi_U \tau(\phi). 
\end{equation}
This implies that the sequence \eqref{exactglobalcomp} is exact. To prove the theorem it therefore remains only to show that the connecting homomorphism
\[
\delta : H^{i-1}(W, P_{k}(L)) \to H^{i+1}(M), \qquad [\alpha] \to [d s(\alpha)]
\]
vanishes. We now see that the splitting $s$ is insufficient because it does not send closed forms to closed forms. Hence, the second step is to modify the splitting so that it does satisfy this property. It is at this stage that the metric $h$ comes into play. 

\begin{lemma} \label{canonicalexactness}
Let $\alpha = \sum_{r = 0}^{k} \alpha_{r} \in \Omega_{W}^{i-1}(P_{k}(L))$ be a closed element. Let $h \in L^{-2}$ be a metric on $L$ and let $\gamma \in \Omega^1_{W}$ be the closed $1$-form such that $\nabla(h) = -2 \gamma \otimes h$. Then 
\[
\tau(\alpha)|_{L \setminus W} = \alpha_0 \wedge \gamma + dv, 
\]
where $v = (-1)^{i}( \sum_{r = 1}^{k} r^{-1} \alpha_r - \frac{1}{2} \log(h) \alpha_0)$ is a horizontal form with singularities along $W$. 
\end{lemma}
\begin{proof}
By Lemma \ref{differentialequation} we have $\epsilon = \frac{1}{2}d \log(h) + \gamma$. Using Equation \ref{eq:principal_diff} and the fact that $\alpha$ is $d^{\sigma}$-closed, we have that $d^{\nabla} \alpha_{p} = p \sum_{r = p+1}^{k} \eta_{r-p} \wedge \alpha_{r}$, for all $p$. Furthermore, viewing $\alpha_{p}$ as a differential form on $L$, we may apply Lemma \ref{differentialequation} to obtain $d \alpha_p = d^{\nabla} \alpha_p - p \epsilon \wedge \alpha_p$. Putting this all together we find 
\[
\tau(\alpha) = \alpha_0 \wedge \epsilon + (-1)^{i} \sum_{r = 1}^{k} \frac{1}{r} ( d^{\nabla} \alpha_{r} - r \epsilon \wedge \alpha_r) = \alpha_0 \wedge \gamma + (-1)^{i} d\big( \sum_{r = 1}^{k} \frac{1}{r} \alpha_r - \frac{1}{2} \log(h) \alpha_0 \big).\hfill \qedhere
\]
\end{proof}
Now define the following modified splitting 
\begin{align} \label{altsplitting}
\mathcal{S}: \ &\Omega_{W}^{i-1}(P_{k}(L)) \to \Omega_{A}^{i}  \\
\sum_{r = 0}^{k} \alpha_{r} \mapsto \psi_U \big( \frac{1}{2} \alpha_{0} \wedge d \log(h) &+ \sum_{r = 1}^{k} \alpha_{r} \wedge \tau_{r} \big) + (-1)^{i} d\psi_{U} \wedge \big( - \frac{1}{2}\log(h) \alpha_{0} + \sum_{r = 1}^{k} r^{-1} \alpha_{r}   \big). \notag
\end{align} 
When $\alpha = \sum_{r = 0}^{k} \alpha_{r}$ is closed, Lemma \ref{canonicalexactness} implies that $\mathcal{S}(\alpha) = d(\psi_U v)$, which is closed. This finishes the proof. \end{proof}

The splitting $\mathcal{S}$ constructed in the proof of Theorem \ref{decomp1proof}, and given by Equation \eqref{altsplitting}, will be used again when proving the results of Section \ref{sec:defconst}. Therefore, we end this section by making a few remarks about its properties. The differential forms in its image are supported in $U$, and it sends closed forms to closed forms. Note also that $\mathcal{S}$ depends on all the chosen data $(\phi, \psi_U, h, \sigma)$. In particular, it depends smoothly on the algebroid $A$ (i.e. on the splitting $\sigma$) through the forms $\tau_r$.

\section{Symplectic geometry on hypersurface algebroids}\label{sec:symalg} 
Now that we have carefully established our formalism for degenerations in the previous sections, we apply it to the study of generically symplectic Poisson structures. We define an \emph{HS symplectic form} to consist of an order $k+1$ hypersurface algebroid $A$ for $(W,M)$ and an algebroid symplectic form $\omega \in \Omega^2_{A}$. More precisely, $\omega$ is an algebroid $2$-form which is closed and non-degenerate, in the sense that the induced map $\omega : A(\sigma) \to A(\sigma)^*$ is an isomorphism. By applying the anchor map to the inverse $\omega^{-1}$, we obtain a Poisson structure $Q = \rho(\omega^{-1})$ on $M$ which is symplectic away from $W$, and degenerates to order $k+1$ along $W$. As a result, it induces a symplectic foliation of the hypersurface. The purpose of the present section is to study the properties of this symplectic foliation. In particular, we establish one of our main results, Theorem \ref{prop:structureonD}, which states that the symplectic variation of the foliation is determined by the extension class of $A$. At the end of this section, we construct HS symplectic forms for the running examples from Section \ref{sec:Examples}. 

The theory developed in this section is local to a neighbourhood of $W$. Hence, we will make use of a tubular neighbourhood embedding and assume that $M = L = \nu_{W}$ for the entirety of this section. We also fix an order $k+1$ HS algebroid $A(\sigma)$ which is determined by a splitting $\sigma = \nabla + \sum_{i = 1}^{k-1}\eta_{i}$. 

\subsection{Local decomposition of HS symplectic structures} Let $\omega \in \Omega^2_{A(\sigma)}$ be an algebroid symplectic form. Using the isomorphism $\tau$ of Proposition \ref{tauchainmap}, the form $\omega$ decomposes into $(\beta, \alpha)$, where $\beta \in \Omega_{L}^2$ is the smooth part and $\alpha = \sum_{r = 0}^{k} \alpha_{r} \in \Omega_{W}^{1}(P_{k}(L))$ is the principal part. The following lemma describes the closure condition in these terms. 

\begin{lemma} \label{lem:sympdata} A closed algebroid $2$-form $\omega \in \Omega^2_{A(\sigma)}$ is equivalent to the following data: 
\begin{enumerate}
\item $\beta \in \Omega_{L}^2$ such that $d\beta = 0$, 
\item $\alpha = \sum_{r = 0}^{k} \alpha_{r} \in \oplus_{r = 0}^{k} \Omega_{W}^{1}(L^{r})$ such that 
\[
d^{\nabla} \alpha_{p} = p \sum_{r = p + 1}^{k} \eta_{r - p} \wedge \alpha_{r},
\]
for all $0 \leq p \leq k$. This is an iterative system of equations starting with 
\[
d^{\nabla}\alpha_{k} = 0, \ \ d^{\nabla}\alpha_{k-1} = (k-1)\eta_{1} \wedge \alpha_k, \ \ d^{\nabla} \alpha_{k-2} = (k-2)(\eta_1 \wedge \alpha_{k-1} + \eta_2 \wedge \alpha_k), \ \ ...
\]
\end{enumerate}
\end{lemma} 
In order for $\omega$ to be non-degenerate in a neighbourhood of the zero section $W$, it suffices to check that the restriction $\omega|_{W}$ to $A(\sigma)|_{W}$ is non-degenerate. This can be analyzed using the restriction map $\mathcal{R}$ from Equation \eqref{restrictionexp}. As a result, we obtain the following description of non-degeneracy. 

\begin{lemma} \label{nondegeneracy} Let $\omega = (\beta, \sum_{r = 0}^{k} \alpha_{r}) \in \Omega_{L}^{2} \oplus \Omega_{W}^{1}(P_{k}(L))$ be an algebroid $2$-form. Its restriction to $W$ is given by $\omega|_{W} = (\gamma, \alpha_{k}) \in \Omega_{W}^2 \oplus \Omega_{W}^{1}(L^k)$, where 
\[
\gamma = \beta|_W + \sum_{r = 1}^{k-1} \alpha_{r} \wedge \eta_{r}. 
\]
Then $\omega$ is non-degenerate in a neighbourhood of $W$ if and only if $\alpha_k$ is nowhere vanishing and $\gamma$ restricts to a non-degenerate $2$-form on the kernel of $\alpha_k$. 
\end{lemma}
\begin{proof}
The description of the restriction follows directly from Equation \eqref{restrictionexp}. Note that the ranks of $\gamma$ and $\alpha_{k}$ (i.e., $\alpha_{k} \wedge \tau_k$) are respectively bounded above by $2n$ and $2$, where the dimension of $W$ is $2n + 1$. Hence, the non-degeneracy of $\omega|_{W}$ is equivalent to $(\omega|_{W})^{n+1} = (n+1) \alpha_k \wedge \gamma^{n}$ being nowhere vanishing. From this the two conditions follow. 
\end{proof}

We can now analyze the geometric structures induced on $W$ from the restriction $\omega|_{W} = (\gamma, \alpha_{k})$. Starting with $\alpha_{k} \in \Omega^{1}_{W}(L^k)$, we first note from Lemma \ref{nondegeneracy} that the kernel $\mathcal{F} =\mathrm{ker}(\alpha_{k}) \subset TW$ defines a corank $1$ distribution on $W$. By Lemma \ref{lem:sympdata}, $\alpha_{k}$ is $d^{\nabla}$-closed and hence $\mathcal{F}$ defines a foliation. In fact, more is true.

\begin{proposition} \label{trivialvariationfoliationinduced}
Let $\nu_{F}$ be the normal bundle of $\mathcal{F}$ in $W$, which is equipped with the flat Bott connection $\nabla^B$. Then $\alpha_k$ defines a flat isomorphism 
\[
\alpha_{k} : (\nu_{\mathcal{F}}, \nabla^B) \to (L^k, \nabla^k)|_{\mathcal{F}}. 
\]
In particular, the foliation $\mathcal{F}$ has trivial variation: $\mathrm{var}(\mathcal{F}) = 0$ (see Definition \ref{def:varfol}). 
\end{proposition}
\begin{proof}
The form $\alpha_k \in \Omega^1_W(L^k)$ defines the foliation $\ff$ and thereby induces an isomorphism of line bundles $\alpha_k: \nu_{\ff} \to L^k$. We must check that it relates the two connections, meaning that for all sections $X \in \mathcal{F}$, $s \in \nu_{\ff}$, we have 
\begin{equation} \label{flatnessofalpha}
\alpha_{k}(\nabla^{B}_{X}(s)) = \nabla_{X}(\alpha_{k}(s)). 
\end{equation}
By definition of the Bott connection $\nabla^{B}_{X}(s) = [X, Y] \text{ mod } \mathcal{F}$, where $Y \in \mathfrak{X}(W)$ is a lift of $s \in \nu_{\ff}$. Because Lemma \ref{flatnessofalpha} implies that $d^{\nabla}\alpha_k = 0$, we find that
\[
d^{\nabla}(\alpha_{k})(X,Y) = \nabla_{X}(\alpha_{k}(Y)) - \nabla_{Y}(\alpha_{k}(X)) - \alpha_{k}([X,Y]) = 0, 
\]
Furthermore, since $X$ is tangent to the foliation, $\alpha_{k}(X) = 0$ and we thus conclude that $\nabla_{X}(\alpha_{k}(Y)) = \alpha_{k}([X,Y])$ from which the desired equation follows. Finally, $\mathcal{F}$ has trivial variation by Proposition \ref{prop:bottextend}. 
\end{proof}

Next, we consider the restriction of $\gamma$ to the foliation $\ff$, denoted $\omega_{\mathcal{F}} = \gamma|_{\mathcal{F}}$. 

\begin{theorem} \label{prop:structureonD}
Let $A(\sigma) \Rightarrow L$ be an order $k+1$ HS algebroid over the total space of a line bundle $L \to W$ and let $e(\sigma) \in H^2(W, L^{-k})$ be its extension class (see Section \ref{extsection}). Let $\omega \in \Omega^2_{A(\sigma)}$ be an algebroid symplectic form, let $\omega|_{W} = (\gamma, \alpha_{k})$ be its restriction to $W$, and let $\mathcal{F}$ be the induced foliation on $W$ from Proposition \ref{trivialvariationfoliationinduced}. Then 
\[
\omega_{\mathcal{F}} = \gamma|_{\mathcal{F}} \in \Omega^{2}_{\mathcal{F}}
\]
is a foliated symplectic form whose symplectic variation is given by the restriction of $e(\sigma)$: 
\[
\mathrm{var}(\omega_{\mathcal{F}}) = -e(\sigma)|_{\mathcal{F}} \in H^{2}_{\mathcal{F}}(\nu^{-1}_{\mathcal{F}}). 
\]
\end{theorem}
\begin{proof}
The decomposition $\omega|_{W} = (\gamma, \alpha_{k})$ is defined using the isomorphism $\Omega_{A(\sigma)|_{W}}^{\bullet} \cong \Omega_{W}^{\bullet} \oplus \Omega_{W}^{\bullet}(L^k)[-1]$ of Equation \eqref{taudecomp}. This is a decomposition as a graded vector space only, but $\Omega_{W}^{\bullet}$ is a subcomplex. Therefore, since $\omega|_{W}$ is closed, we have the equation 
\begin{equation*}
d\gamma = -d_{A(\sigma)|_{W}} \alpha_{k}.
\end{equation*}
We can compute the right hand side in $\Omega_{W}^{\bullet}(P_{k}(L))[-1]$ and then apply the restriction map $\mathcal{R}$. Hence, recalling that $d^{\nabla} \alpha_{k} = 0$, we obtain
\begin{align*}
d \gamma &=  -d_{A(\sigma)|_{W}} \alpha_{k} = - \mathcal{R}( d^{\sigma} \alpha_{k}) =  \mathcal{R}( \alpha_{k} \wedge \sum_{i = 1}^{k-1} (i-k) \eta_{i} \wedge \tau_{k-i} ) \\
&= \alpha_{k} \wedge \sum_{i = 1}^{k-1} (i-k) \eta_{i} \wedge \eta_{k-i} = -e(\sigma) \wedge \alpha_{k}. 
\end{align*}
This implies, in particular, that $d_{\mathcal{F}}\omega_{\mathcal{F}} = 0$. Hence, since $\omega_{\mathcal{F}}$ is non-degenerate by Lemma \ref{nondegeneracy}, it defines a foliated symplectic form. Furthermore, since $\gamma$ is an extension of $\omega_{\mathcal{F}}$ to $\Omega_{W}^2$, it can be used to compute the symplectic variation, giving us 
\[
\mathrm{var}(\omega_{\mathcal{F}}) = \frac{d \gamma}{\alpha_{k}} = -e(\sigma)|_{\ff}.\hfill\qedhere
\]
\end{proof}
\begin{remark}
In the case of $b$ and $b^m$-symplectic structures, where the splitting $\sigma$ is trivial in the sense that all $\eta_i = 0$, the extension class vanishes. As a result, Theorem \ref{prop:structureonD} implies that the symplectic variation must also vanish, recovering the results of \cite{GUILLEMIN2014864, MR3523250}. In contrast, we see here that the symplectic variation of general HS symplectic structures can be non-zero, reflecting the obstruction to lifting $G_{k}$-representations to $G_{k+1}$. 
\end{remark}

\subsection{Example: Arnold's cat map} \label{cat2} We continue the example from Section \ref{cat}. This will provide a $6$-dimensional example of an HS symplectic form with compact hypersurface $W$ and non-vanishing symplectic variation. 

By Lemma \ref{lem:sympdata}, in order to construct an algebroid symplectic form $\omega \in \Omega^2_{A}$, we need the data of a closed $2$-form $\beta \in \Omega^2_{L}$, which we will assume is pulled back from $W$, and a closed $1$-form in $\Omega^{1}_{W}(P_{k}(L))$. By Theorem \ref{lambdacohomology}, $H^{2}_{W}$ is the eigenspace of $\phi_{*} : H^{2}(\mathbb{T}^4) \to H^{2}(\mathbb{T}^4)$ for the eigenvalue $1$. Hence $H^{2}_{W} \cong (E_{r} \otimes E_{r^{-1}}) \oplus (E_{r^2} \otimes E_{r^{-2}})$. This contains an element with the following representative 
\[
\beta = a_{1} \wedge b_{1} + a_{2} \wedge b_{2}. 
\]
Let $t_3$ be the unit section of the trivial bundle $L^3$. As a function on $L = W \times \mathbb{R}$ it is given by $t_{3} = t^{-3}$, where $t$ is the linear coordinate on $\mathbb{R}$. We define 
\[
\omega = \beta - d \tau_{3}(t_{3}) \in \Omega^2_{A}. 
\]
In terms of the decomposition of Proposition \ref{tauchainmap}, the smooth part of $\omega$ is $\beta$, and the higher residues are 
\[
\alpha_{3} = 3 \lambda \alpha, \qquad \alpha_2 = 2e^{-\lambda \theta} a_1, \qquad \alpha_1 = e^{-2 \lambda \theta} a_{2}, \qquad \alpha_0 = 0. 
\]
Written as a singular form on $W \times \mathbb{R}$, we have 
\[
\omega = a_{1} \wedge b_{1} + a_{2} \wedge b_{2} + ( 3 \lambda \alpha + 2 t e^{- \lambda \theta} a_1+ t^2e^{-2\lambda \theta} a_2) \wedge \frac{dt}{t^4} + ( \frac{e^{-\lambda \theta} }{t^2} a_{1} + \frac{2 e^{-2\lambda \theta}}{t} a_2 ) \wedge \lambda \alpha.  
\]
The restriction of $\omega$ to $W$ is given by 
\[
\omega|_{W} = (\gamma = \beta + e^{-3 \lambda \theta} a_{1} \wedge a_{2}, 3 \lambda \alpha t_{3}).
\]
The $1$-form $3 \lambda \alpha$ induces the foliation by fibres of the projection map $f: W \to S^1$, and $\gamma$ restricts to the following foliated symplectic form
\[
\omega_{\mathcal{F}} = a_{1} \wedge b_{1} + a_{2} \wedge b_{2} +  e^{-3 \lambda \theta} a_{1} \wedge a_{2}.
\]
Therefore, $\omega$ is non-degenerate in a neighbourhood of $W$ by Lemma \ref{nondegeneracy}. The symplectic variation is non-zero and given by $\mathrm{var}(\omega_{\mathcal{F}}) = -e^{-3 \lambda \theta} a_{1} \wedge a_{2}$, viewed as an element of the foliated cohomology. 

Finally, writing out the Poisson structure $\pi = \omega^{-1}$ explicitly, we find
\[
\pi = B_1 \wedge A_1 + B_2 \wedge A_2 + \frac{t^4}{3 \lambda} \partial_t \wedge \partial_{\theta} + e^{-\lambda \theta} B_1 \wedge ( \frac{t^2}{3} \partial_t - \frac{2t}{3 \lambda} \partial_{\theta}) + e^{-2\lambda \theta} B_2 \wedge (\frac{2t^3}{3} \partial_t - \frac{t^2}{3 \lambda} \partial_{\theta}) + e^{- 3 \lambda \theta} B_1 \wedge B_2, 
\]
where $A_{1}, B_{1}, A_{2}, B_{2} \in \mathfrak{X}(\mathbb{T}^4)$ are constant vector fields on the $4$-torus $\mathbb{T}^4$ which are dual to the basis $a_{1}, b_{1}, a_{2}, b_{2}$. 

\subsection{Example: Universal algebroid} \label{subsec:UniversalHSAlg} Recall the equivariant order $2(n+1)$ hypersurface algebroid $E_{2n+1} \Rightarrow G_{2n+1} \times \mathbb{R}$ from Section \ref{sec:equivuniversalalg}. In this section we will construct an HS symplectic form on this algebroid which provides a `Poisson avatar' for the group extension 
\begin{equation} \label{eq:groupseq}
1 \to M_{2n+1} \to G_{2n+2} \to G_{2n+1} \to 1. 
\end{equation}
Morally, at least, it is also responsible for our deformation construction in the exact case (Theorem \ref{exactdef}). 

In order to construct the HS symplectic form, we can work in the universal cdga of Equation \eqref{univcdga}
\[
S_{2n+1} = \mathbb{R}[x_{0}, ..., x_{2n}]\langle t_{0}, ..., t_{2n+1}\rangle
\]
since, by Proposition \ref{tauchainmap}, it corresponds to a subalgebra of the $G_{2n+1}$-invariant algebroid de Rham complex of $E_{2n+1}$. Therefore, the following exact form 
\begin{equation}\label{eq:univsymp}
\varpi_{2n+1} = dt_{2n+1} = \sum_{i = 0}^{2n} (i-1-2n) x_{i} t_{2n+1-i}
\end{equation}
defines a closed $2$-form for $E_{2n+1}$.

\begin{proposition}\label{prop:univsymp} 
The form $\varpi_{2n+1} \in \Omega^2_{E_{2n+1}}$ defines a $G_{2n+1}$-invariant order $2(n+1)$ HS symplectic form for $(G_{2n+1}, G_{2n+1} \times \rr)$. The induced foliation $\mathcal{F}$ on $G_{2n+1}$ is given by the fibres of the homomorphism $\pi_{2n+1} : G_{2n+1} \to G_1$, and the foliated symplectic form is given by the restriction to $\mathcal{F}$ of 
\[
\gamma = \sum_{i = 1}^{2n} i x_i \wedge x_{2n+1 - i} \in \mathrm{CE}^2(\mathfrak{g}_{2n+1}).
\]
\end{proposition}
\begin{proof}
Restricting $\varpi_{2n+1}$ to the hypersurface $G_{2n+1}$ using Lemma \ref{nondegeneracy}, we get $\varpi_{2n+1}|_{G_{2n+1}} = (\gamma, \alpha_{2n+1})$, where $\alpha_{2n+1} = -(2n+1)x_0 t_{2n+1}$. The differential $d\pi_{2n + 1}$ evaluates to $1$ on $V_0 \in \mathfrak{X}(G_{2n+1})^{R}$ and vanishes on $V_{i} \in \mathfrak{X}(G_{2n+1})^{R}$ for $i > 0$. Hence, it can be identified with $x_0$. This implies that the leaves of the foliation $\mathcal{F} = \mathrm{ker}(\alpha_{2n+1})$ are given by the fibres of $\pi_{2n+1}$. 

Next, expanding the sum, we find 
\begin{equation} \label{explicitexpression}
\gamma = (1-2n)x_{1} \wedge x_{2n} + (3 - 2n)x_{2} \wedge x_{2n-1} + ... + (-1) x_{n} \wedge x_{n+1}. 
\end{equation}
Since $x_{0}, x_{1}, ..., x_{2n}$ form a global basis of right invariant $1$-forms on $G_{2n+1}$, it follows that $\gamma$ restricts to a symplectic form on $\mathrm{ker}(\alpha_{2n+1})$. By Lemma \ref{nondegeneracy}, this implies that $\varpi_{2n+1}$ is non-degenerate in a neighbourhood of the hypersurface $G_{2n+1}$. To see that $\varpi_{2n+1}$ is globally non-degenerate, recall from Lemma \ref{equivactionEk} that $E_{2n+1}$ admits a left $\mathbb{R}^*$-action lifting $L_{r}(g,t) = (rg, rt)$ on $G_{2n+1} \times \mathbb{R}$. As already explained, this action shows up as a weight grading on $S_{2n+1}$ for which $t_{2n+1}$, and hence $\varpi_{2n+1}$, has weight $-(2n+1)$. As a result, 
\begin{equation}\label{eq:anactionaday}
L_{r}^*(\varpi_{2n+1}) = r^{-2n-1}\varpi_{2n+1}. 
\end{equation}
Since any neighbourhood of $G_{2n+1}$ intersects all orbits of this action, it follows that $\varpi_{2n+1}$ is non-degenerate everywhere. 
\end{proof}

Recall from our discussion in Section \ref{sec:equivuniversalalg} that $E_{2n+1}$ is isotopic to a Scott $b^{2n+2}$-algebroid because the connected components of $G_{2n+1}$ are contractible. This implies that the extension class of $E_{2n+1}$ is trivial. On the other hand, as we have seen, if we consider the extension class in the $G_{2n+1}$-invariant cohomology, then it \emph{is non-zero} and can be identified with the extension class $e_{2n+1}$ of \eqref{eq:algebraseq}, the short exact sequence of Lie algebras which differentiates the above extension of groups \eqref{eq:groupseq}. We can realize this extension class \emph{geometrically} by taking a quotient of $G_{2n+1}$ by the subgroup $K_{2n+1}(\mathbb{Z})$ considered in Section \ref{sec:universalalg}. 

Consider the right action of $K_{2n+1}(\mathbb{Z})$ on $G_{2n + 1} \times \mathbb{R}$. The quotient is given by $G_{2n+1}/K_{2n+1}(\mathbb{Z}) \times \mathbb{R}$. Furthermore, using the left $\mathbb{R}^*$-action on $G_{2n+1}$, we obtain an isomorphism 
\[
G_{2n+1}/K_{2n+1}(\mathbb{Z}) \cong \mathbb{R}^* \times X_{2n+1},
\] 
where $X_{2n+1}$ is the classifying space for $K_{2n+1}(\mathbb{Z})$ considered in Section \ref{sec:universalalg}. Therefore, 
\[
(G_{2n + 1} \times \mathbb{R})/K_{2n+1}(\mathbb{Z}) \cong (\mathbb{R}^* \times X_{2n+1}) \times \mathbb{R}. 
\]
Because the HS symplectic form $\varpi_{2n+1}$ is $G_{2n+1}$-invariant, it descends to $(\mathbb{R}^* \times X_{2n+1}) \times \mathbb{R}$. The induced foliation on $(\mathbb{R}^* \times X_{2n+1})$ is given by the projection to $\mathbb{R}^*$, and the symplectic foliation is given by 
\[
\omega_{\mathcal{F}} = u^{-2n-1} \Omega_{2n+1},
\]
where $u$ is the coordinate on $\mathbb{R}^*$, and 
\[
\Omega_{2n+1} =  \sum_{i = 1}^{2n} i x_i \wedge x_{2n+1 - i}  \in \Omega^2_{X_{2n+1}}.
\]
The symplectic variation is given by $\mathrm{var}(\omega_{\mathcal{F}}) = - u^{-2n-1} k_{2n+1}$, where $k_{2n+1} \in H^2(X_{2n+1})\cong H^2(\mathfrak{k}_{2n+1})$ is the extension class of \eqref{eq:algebraseq2}, which is non-zero for $n\geq 1$. 
\begin{remark}
The symplectic forms $\Omega_{2n + 1} \in  \Omega^2_{X_{2n+1}}$ were considered in \cite{babenko1998nonformal, babenko2000massey}, where they were used to construct the first examples of non-formal simply connected symplectic manifolds in dimensions greater than $8$. Here we see that they arise as symplectic leaves of a family of Poisson structures on $X_{2n + 1} \times \mathbb{R}^* \times \mathbb{R}$, for $n \geq 1$. 
\end{remark}

\section{Deformation construction}\label{sec:defconst} 
In this section we turn to our second main result: deformations of a hypersurface algebroid can be lifted to deformations of an HS symplectic form, uniquely up to isotopy. Since hypersurface algebroids are classified by $G_{k}$-representations of $\pi_1(W)$ (Theorem \ref{RHthm}), this will be applied in Section \ref{sec:deformalongcharvar}  to construct maps from the $G_{k}$-character varieties into the moduli space of Poisson structures. 

The following assumption on HS algebroids helps to control their cohomology and will be required for the uniqueness part of our results: 
\begin{definition}
A hypersurface algebroid for $(W, M)$ is called \emph{nice} if the induced local system on $L = \nu_{W}$ does not have its monodromy contained in $\{ \pm 1 \}$. 
\end{definition}

The first deformation result we obtain is the following: 
\begin{theorem}[Deformation construction: general case] \label{th:gendef} Let $A$ be an order $k+1$ HS algebroid for $(W,M)$ such that $M$ is compact, and let: 
\begin{itemize}
\item $\omega \in \Omega^2_{A}$ be an algebroid symplectic form, 
\item $A(t)$ be a family of HS algebroids deforming $A$, 
\item $r(t) \in \Omega_{W}^{1}(P_{k}(L))$ be a smooth family of closed $1$-forms with $[r(0)] = \mathrm{Princ}[\omega] \in H^{1}_{W}(P_{k}(L))$. 
\end{itemize}
Then for small $t$, there is a family of algebroid symplectic forms $\omega(t) \in \Omega^2_{A(t)}$ such that $\omega(0) = \omega$ and $\mathrm{Princ}[\omega(t)] = [r(t)]  \in H^{1}_{W}(P_{k}(L))$. Furthermore, such a family can be chosen to agree with $\omega$ away from a tubular neighbourhood of $W$. 

Finally, suppose $\omega^1(t)$ and $\omega^2(t)$ are two such families such that their cohomology classes agree for small $t$: 
\[
[\omega^1(t)] = [\omega^2(t)] \in H^2(A(t)).
\]
If the algebroid $A$ is nice, then for small $t$, there is a family of Lie algebroid automorphisms $\psi(t) \in \mathrm{Aut}(A(t))$ such that 
\[
\psi(t)^{*} (\omega^2(t)) = \omega^1(t). 
\]
\end{theorem}

To our knowledge, the compactness condition on the ambient manifold $M$ is required to produce the automorphisms in Theorem \ref{th:gendef}. The following Theorem simplifies the deformation result at the cost of an extra condition on the form $\omega$. 

\begin{theorem}[Deformation construction: exact case] \label{exactdef} Let $A$ be an order $k + 1$ HS algebroid for $(W,M)$ such that $W$ is compact and let: 
\begin{itemize}
\item $\omega \in \Omega^2_{A}$ be an algebroid symplectic form which is cohomologous to a smooth $2$-form on $M$, 
\item $A(t)$ be a family of Lie algebroids deforming $A$. 
\end{itemize}
Then for small $t$, there is a family of algebroid symplectic forms
\[
\omega(t) = B + d\alpha(t) \in \Omega^2_{A(t)}, 
\]
such that $\omega(0) = \omega$, where $B \in \Omega^2_{M}$ is a smooth closed $2$-form and $\alpha(t) \in \Omega^1_{A(t)}$ is supported in an arbitrarily small compact tubular neighbourhood of $W$. 

Furthermore, if the algebroid $A$ is nice, then given two such families $\omega^1(t)$ and $\omega^2(t)$, for small $t$ there is a family of automorphisms $\psi(t) \in \mathrm{Aut}(A(t))$ which are compactly supported around $W$, such that $\psi(0) = id$, and such that 
\[
\psi(t)^*(\omega^2(t)) = \omega^1(t).
\]
\end{theorem}

Finally, our last result states that the deformation construction only depends on $A(t)$ up to isotopy. 

\begin{corollary}[Isotopy invariance] \label{isotopyinv} Let $A$ be a nice order $k+1$ HS algebroid for $(W,M)$ such that $W$ is compact. Let $\omega$ be an $A$-symplectic form which is cohomologous to a smooth $2$-form on $M$. Let 
\[
(A^{i}(t), \omega^{i}(t) = B^{i} + d\alpha^{i}(t)), \qquad i = 1, 2, 
\]
be two deformations of $(A, \omega)$ as symplectic HS algebroids such that $B^{i} \in \Omega^2(M)$ are smooth closed $2$-forms and $\alpha^{i}(t) \in \Omega^1_{A^{i}(t)}$ are supported in small compact tubular neighbourhoods of $W$. Suppose that the families of algebroids are isotopic, meaning that there is a $t$-family of isotopies $\phi_{t} : A^{1}(t) \to A^{2}(t)$ such that $\phi_{0} = id$. Then the two symplectic algebroids are isotopic\footnote{Note however that these isotopies need not restrict to the identity on $W$} for small $t$. 
\end{corollary}

\subsection{Existence Proofs}
In this section we prove the existence portions of Theorems \ref{th:gendef} and \ref{exactdef}. This relies on the short exact sequence of Theorem \ref{decomp1proof} and the splitting $\mathcal{S}$ \eqref{altsplitting} produced in its proof. The main idea is to use this splitting to write an HS symplectic structure in a `prepared form', consisting of a smooth part which stays constant, and a singular part that gets deformed. 

Recall that $\mathcal{S}$ varies smoothly as the Lie algebroid $A$ gets deformed, it sends closed forms to closed forms, and the forms in its image are supported in an arbitrary tubular neighbourhood $U$, which we take to be precompact. The main result we will need is the following lemma. 

\begin{lemma} \label{deformationone}
Let $U$ be a tubular neighbourhood of $W$. Suppose a closed algebroid form $\omega \in \Omega^2_{A}$ is cohomologous to a smooth form on $M$. Then there exists $B \in \Omega^{2,cl}_{M}$ and $\alpha \in \Omega^1_{A}$ supported in $U$ such that 
\[
\omega = B + d\alpha.
\] 
Furthermore, given a deformation $A(t)$ of $A$, there exists a family $\alpha(t) \in \Omega^1_{A(t)}$ supported in $U$ with $\alpha(0) = \alpha$. Consequently, the family 
\[
\omega(t) = B + d \alpha(t) \in \Omega^2_{A(t)}
\]
is a deformation of $\omega$ by closed algebroid forms. 
\end{lemma}
\begin{proof}
By assumption there exists $C \in \Omega^2_{M}$ and $\gamma \in \Omega^1_{A}$ such that $\omega = C + d \gamma$. Let $\alpha = \mathcal{S}(\mathrm{Princ}(\gamma)) \in \Omega^1_{A}$, which is an algebroid $1$-form supported in $U$ such that $\mathrm{Princ}(\gamma) = \mathrm{Princ}(\alpha)$. Let $B = C + d(\gamma- \alpha) \in \Omega^2_M$, which is smooth. Then $\omega = B + d\alpha$. Moreover, given a deformation $A(t)$ of $A$, we define $\alpha(t) = \mathcal{S}_t(\mathrm{Princ}(\gamma)) \in \Omega^1_{A(t)}. $
\end{proof}

 \begin{proof}[Proof of Theorem \ref{th:gendef}: Existence part] Given the closed form $r(0) \in \Omega_{W}^1(P_{k}(L))$, let $\kappa = \omega - \mathcal{S}(r(0)) \in \Omega^{2,cl}_{A}$. This form is closed and 
\[
\mathrm{Princ}([\kappa]) = \mathrm{Princ}([\omega]) - [r(0)] = 0 \in H^1(P_{k}(L)). 
\]
Hence, by the exactness of the cohomology sequence from Theorem \ref{decomp1proof}, $[\kappa] \in H^2(M)$, meaning that $\kappa$ is cohomologous to a smooth form. Hence, by Lemma \ref{deformationone}, there is a smooth closed form $B \in \Omega^{2,cl}_{M}$ and an algebroid form $\alpha \in \Omega_{A}^1$ supported in $U$ such that $\kappa = B + d\alpha$. In addition, given the deformation $A(t)$, Lemma \ref{deformationone} also provides a family of algebroid forms $\alpha(t) \in \Omega^1_{A(t)}$ supported in $U$ with $\alpha(0) = \alpha$, so that 
\[
\kappa(t) = B + d\alpha(t) \in \Omega^{2,cl}_{A(t)}
\]
is a family of closed forms deforming $\kappa$. Now define 
\[
\omega(t) = \kappa(t) + \mathcal{S}_t(r(t)) = B + \mathcal{S}_t(r(t)) + d\alpha(t) \in  \Omega^{2,cl}_{A(t)}.
\]
Then $\omega(0) = \kappa + \mathcal{S}(r(0)) = \omega$ and $\mathrm{Princ}([\omega(t)]) = [r(t)]$ as desired. Finally, since both $\mathcal{S}_t(r(t))$ and $\alpha(t)$ are supported in $U$, $\omega(t)$ agrees with $\omega$ outside of $U$. Since $\omega$ is non-degenerate and $M$ is compact, $\omega(t)$ will also be non-degenerate, and hence symplectic, for small $t$. 
\end{proof}

\begin{proof}[Proof of Theorem \ref{exactdef}: Existence part] We apply Lemma \ref{deformationone} to obtain the deformation $\omega(t) = B + d\alpha(t) \in \Omega^2_{A(t)}$. Because $\alpha(t)$ has compact support, the form $\omega(t)$ will remain symplectic for small $t$.
\end{proof}

\subsection{Uniqueness Proofs} In this section we prove the uniqueness portions of Theorems \ref{th:gendef} and \ref{exactdef}, as well as Corollary \ref{isotopyinv}. The idea is again to use the splitting $\mathcal{S}$, this time in order to bring the HS symplectic structures into a form suitable for an application of the Moser argument. The condition that the HS algebroid $A$ be `nice' is useful because it implies the vanishing of cohomology groups, which helps us to choose primitives. 

We start by stating the cohomological consequences of niceness. Observe that there is a vector bundle decomposition $P_{k}(L) = L^0 \oplus P_{k}^{+}(L)$ which is compatible with the principal representation. As a result, $\Omega_W^{\bullet}(P_{k}(L)) = \Omega_{W}^{\bullet} \oplus \Omega_{W}^{\bullet}(P_{k}^+(L))$ is a decomposition of cochain complexes. The following result follows from the fact that when $A$ is nice, the connections on $L^{i}$, for $i > 0$, and hence on $P_{k}^{+}(L)$, admit no flat sections. 
\begin{lemma} \label{myobstruction}
Let $A$ be a nice order $k+1$ HS algebroid. Then 
\[
H^{0}(W, P_k(L)) = H^{0}(W) = \mathbb{R}.
\]
\end{lemma}

The following lemma generalizes the well-known fact that smooth families of exact $1$-forms have smooth families of primitives. It relies crucially on the fact that this result also holds in the setting of a family of flat connections $\nabla_t$ on a line bundle $\Lambda$, such that the monodromy is non-trivial for all $t$. 
\begin{lemma} \label{smoothprim} 
Let $A$ be a nice order $k+1$ HS algebroid with deformation $A(t)$, and let $\xi(t) \in \Omega_{W}^1(P_{k}(L))$ be a smooth family of $1$-forms which are exact with respect to the family of differentials $d^{\sigma(t)}$. Then there exists a smooth family of primitives $u(t) \in \Omega_{W}^0(P_{k}(L))$ for small values of $t$, in the sense that 
\[
d^{\sigma(t)} u(t) = \xi(t).
\] 
\end{lemma}
\begin{proof}
We can write $\xi(t) = \xi_0(t) + \xi^{+}(t)$ using the above decomposition of $P_{k}(L)$. The smooth family of primitives for $\xi_0(t) \in \Omega_{W}^1$ is produced in the standard way. To produce the primitives of $\xi^+(t) = (\xi_1(t), ..., \xi_k(t))$ we work recursively, starting with $\xi_k(t)$. Since $A$ is nice, the monodromy of $\nabla_t$ and its positive powers are non-trivial for small $t$. Since $\xi^+(t)$ is exact, it has a primitive $s(t) = (s_{1}(t), ..., s_{k}(t))$. Looking at the formula for $d^{\sigma(t)}$ shows that $\xi_k(t) = d^{\nabla^{k}_{t}} s_{k}(t)$ is exact for all $t$, and hence we can assume that $s_k(t)$ is smooth in $t$. Working inductively then completes the proof. 
\end{proof}

 \begin{proof}[Proof of Theorem \ref{th:gendef}: Uniqueness part]
 Using Lemma \ref{smoothprim} and the family of splittings $\mathcal{S}_{t}$, we can find a smooth family of algebroid forms $\gamma(t) \in \Omega^1_{A(t)}$ such that $\omega^2(t) - \omega^1(t) = d\gamma(t)$. We now apply the Moser argument. 
%
%
%
 \end{proof}
 
 \begin{proof}[Proof of Theorem \ref{exactdef}: Uniqueness part]
 Consider two deformations $\omega^i(t) = B^{i} + d \alpha^{i}(t)$, for $i = 1,2$, as in the Theorem statement. Note here that in general $B^1 \neq B^2$. The goal is to re-express these deformations in a uniform way as 
\[
\omega^1(t) = B + d \eta^1(t), \qquad \omega^2(t) = B + d\eta^2(t),
\]
where $\eta^1(t), \eta^2(t) \in \Omega^1_{A(t)}$ are families of algebroid forms supported in a \emph{precompact} tubular neighbourhood of $W$, and such that $\eta^1(0) = \eta^2(0)$. Once this is done, we can then apply the standard Moser argument because the vector field we produce will have compact support and is therefore complete. 

At $t = 0$, $\mathrm{Princ}(\omega) = d^\sigma \mathrm{Princ}(\alpha^i(0))$ for $i = 1,2$. Let $\beta = \mathrm{Princ}(\alpha^2(0) - \alpha^1(0)) \in P_{k}(L)$, which is $d^\sigma$-closed. By Lemma \ref{myobstruction}, $\beta \in L^0$ is a constant. Therefore, $d^{\sigma(t)} \beta = 0$ for all $t$, and hence $\mathcal{S}_{t}(\beta) \in \Omega^{1,cl}_{A(t)}$ is a family of closed $1$-forms supported in a precompact tubular neighbourhood $U$. Replacing $\alpha^1(t)$ with $\alpha^1(t) + \mathcal{S}_{t}(\beta) \in \Omega^1_{A(t)}$, we can still write $\omega^1(t) = B^1 + d \alpha^1(t)$. But now $\mathrm{Princ}(\alpha^2(0) - \alpha^1(0)) = 0$, so that $\xi = \alpha^2(0) - \alpha^1(0) \in \Omega^1_M$ is a smooth $1$-form supported in $U$. Furthermore, from the equation $\omega^1(0) = \omega = \omega^2(0)$, we conclude that $B^1 = B^2 + d \xi$. Now define 
\[
B = B^2 \in \Omega^{2,cl}_{M}, \qquad \eta^1(t) = \alpha^1(t) + \xi \in \Omega^1_{A(t)}, \qquad \eta^2(t) = \alpha^2(t) \in \Omega^1_{A(t)}. 
\]
With these forms, the deformations are expressed in a uniform way as desired. 
\end{proof}

\begin{proof}[Proof of Corollary \ref{isotopyinv}]
 First, note that we can assume that $\phi_{t}$ is supported in a compact tubular neighbourhood $U$ of $W$. By pulling back by $\phi_{t}$, the family $(A^2(t), \omega^2(t))$ is isomorphic to 
 \[
 (A^1(t), \phi_{t}^*(\omega^2(t)) = \phi_{t}^*(B^2) + d \phi_{t}^*(\alpha^2_{t})).
 \]
 Note that $B^2$ is a smooth closed $2$-form. Let $V_{t} \in \mathfrak{X}(M)$ be the time-dependent vector field on $M$ obtained by differentiating the family $\phi_{t}$. Then $\phi_{t}^*(B^2) = B^2 + d \gamma(t)$, where $ \gamma(t) = \int_{0}^{t} \phi_{s}^*(\iota_{V_s}B^2) ds \in \Omega^1(M). $
 Note that since $\phi_{s}$ and $V_{s}$ are both supported in the compact neighbourhood $U$, so is $\gamma(t)$. Therefore, 
 \[
 \phi_{t}^*(\omega^2(t)) = B^2 + d \beta(t), 
 \]
 where $\beta(t) = \phi_{t}^*(\alpha_{t}^2) + \gamma \in \Omega^1_{A^1(t)}$ is an algebroid $1$-form supported in a compact tubular neighbourhood. Furthermore, $\phi_{0}^*(\omega^2(0)) = \omega$. We are now in the setting of Theorem \ref{exactdef}. Hence, there is a family of automorphisms $\psi(t) \in \mathrm{Aut}(A^1(t))$ such that $\psi(t)^* \phi_{t}^*(\omega^2(t)) = \omega^1(t),$ for small $t$. 
 \end{proof}

\section{Deforming along the character variety} \label{sec:deformalongcharvar}
Taken together, Theorem \ref{exactdef} and Corollary \ref{isotopyinv} tell us that, given a nice symplectic HS algebroid $(A, \omega)$ with the property that $\omega$ is cohomologous to a smooth form, deformations of $A$ can be lifted to deformations of $\omega$, and that, up to isomorphism, such deformations depend only on the isotopy class of the algebroid deformation. Recall from Section \ref{classificationreminder} that the space of HS algebroids up to isotopy was classified in \cite{foliationpaper} and shown to be homeomorphic to a character variety. Therefore, this suggests that the pair $(A, \omega)$ should determine a map from a character variety to the space of Poisson structures. 

Furthermore, recall from Section \ref{extsection} that an HS algebroid has an extension class, which determines the symplectic variation of any algebroid symplectic form by Theorem \ref{prop:structureonD}. We therefore have the following schematic commutative diagram. 

\begin{center}
\begin{tikzcd}
\text{Character Variety} \ar[r] \ar[d] &  \text{Poisson Structures}  \ar[d] \\
\text{Extension Class} \ar[r] & \text{Symplectic Variation}
\end{tikzcd}
\end{center}
The fact that the variation of the symplectic foliation can be deduced from a characteristic class of the HS algebroid gives us a straightforward method of showing that the families of Poisson structures we produce are non-constant.  

In this section, we implement this construction for a certain class of symplectic HS algebroids $(A, \omega)$ constructed from the data of cosymplectic structures on the hypersurface $W$. 

\subsection{The construction}
\subsubsection{Cosymplectic geometry}  \label{sec:cosymplectic}
A cosymplectic structure on a compact manifold $W$ of dimension $2n + 1$ consists of a pair of forms $(a, B)$, where $a \in \Omega^{1,cl}_{W}$ is a closed nowhere vanishing $1$-form and $B \in \Omega^{2,cl}_{W}$ is a closed $2$-form, such that 
\[
B^{n} \wedge a \in \Omega^{2n+1}_{W}
\]
is a volume form. Given this data, consider the trivial bundle $L = W \times \mathbb{R}$ along with the singular symplectic form 
\[
\omega = B - ka \wedge \frac{dt}{t^{k+1}},
\]
where $t$ denotes the linear coordinate along $\mathbb{R}$. At first glance, it may appear most natural to view this as an algebroid differential form for the HS algebroid associated to the trivial splitting $\sigma = d$ (i.e., a Scott $b^{k+1}$-algebroid). However, this algebroid is not \emph{nice} because the monodromy of the induced local system on $L$ is trivial, and the form $\omega$ is not cohomologous to a smooth form since $a$ is not exact. 

Instead, consider the HS algebroid $A(a)$ associated to the splitting $\sigma = \nabla$, where $\nabla = d - a$ is a flat connection on $L$. Because the cohomology class of $a$ is non-trivial, the algebroid $A(a)$ is nice. This algebroid is generated by the following vector fields 
\[
X + a(X) t \partial_{t}, \qquad t^{k+1} \partial_{t},
\]
where $X \in \mathfrak{X}(W)$. By Corollary \ref{localbasiscor}, the dual space $A(a)^*$ is generated by the $1$-forms 
\[
\alpha, \qquad \tau_{k}(t_{k}) = \frac{dt}{t^{k+1}} - \frac{a}{t^{k}},
\]
where $\alpha \in \Omega^1(W)$, and where $t_{k} \in L^{k}$ is the unit section. The singular symplectic form is then given by 
\[
\omega = B + d \tau_{k}(t_{k}) \in \Omega^2_{A(a)},
\]
which is cohomologous to a smooth form. The data $(A(a), \omega)$ is thus ready to be deformed. 

\subsubsection{Deformation} Let $\sigma = \nabla + \sum_{i= 1}^{k-1} \eta_{i}$ be a splitting as in Theorem \ref{MCequation}. For any real number $s \in \mathbb{R}$, the following also defines a splitting 
\[
\sigma(s) = \nabla + \sum_{i= 1}^{k-1} s^{i} \eta_{i}.
\]
Hence, we obtain a family of order $k+1$ HS Lie algebroids $A(\sigma(s))$ deforming $A(\sigma(0)) = A(a)$. Applying Theorem \ref{exactdef}, for small $s$ we get a family of algebroid symplectic forms $\omega(s) \in \Omega^2_{A(\sigma(s))}$ such that $\omega(0) = \omega$. In a neighbourhood of $W$, this form has the following expression 
\[
\omega(s) = B + \tau_{s}(d^{\nabla} t_{k} + \sum_{i = 1}^{k-1} (i-k)s^{i} \eta_{i} t_{k - i}), 
\]
where $\tau_{s}$ is the map of Proposition \ref{tauchainmap}, which is $s$-dependent, and where $t_{i} \in L^{i}$ is the unit section. By Lemma \ref{nondegeneracy}, the foliation of $W$ induced by $\omega(s)$ is given by $\mathcal{F} = \mathrm{ker}(a)$, and $\alpha_{k} = -k \ a \ t_{k} : (\nu_{\mathcal{F}}, \nabla^{B}) \to (L^k, \nabla^k)|_{\mathcal{F}}$ is the isomorphism between the Bott connection and the $k^{th}$-power of $\nabla$. Furthermore, the foliated symplectic form induced by $\omega(s)$ is given by 
\[
\omega_{\mathcal{F}}(s) = B - s^k \sum_{r = 1}^{k-1} r \eta_{k-r} \wedge \eta_{r}, 
\]
which has symplectic variation 
\[
\mathrm{var}(\omega_{\mathcal{F}}(s)) = - e(\sigma(s)) = - s^{k} e(\sigma) = - s^k [ \sum_{r = 1}^{k-1} r \eta_{k-r} \wedge \eta_{r}] \in H^{2}_{\mathcal{F}}(\nu^{-1}_{\mathcal{F}}). 
\]

\subsubsection{Character Variety} We briefly recall the classification of HS algebroids from Section \ref{classificationreminder}. There is a natural map 
\[
\mathrm{Hom}(\pi_{1}(W), G_{k}) \to H^{1}(W, \mathbb{Z}/2) \times H^{1}(W, \mathbb{R}),
\]
where the two terms on the right respectively classify line bundles and flat connections. Denote the pre-image of $(L, [a])$ by $\mathrm{Hom}_{(L, [a])}(\pi_{1}(W), G_{k})$. This is the fibre corresponding to the flat line bundle $(L, \nabla)$. Then we define the following $G_{k}$-character variety of $(W, L, a)$
\[
M_{k}(L, W, a) = \mathrm{Hom}_{(L, [a])}(\pi_{1}(W), G_{k})/G_{k}^{0}.
\]
By \cite[Corollary 8.42]{foliationpaper}, this space is homeomorphic to the space of order $(k+1)$ hypersurface algebroids for $(W,L)$ which induce the flat connection $\nabla$ on $L$, modulo isotopies which restrict to the identity along $W$. 

Our goal in this section is to show that the symplectic algebroids $(A(\sigma(s)), \omega(s))$ constructed above depend only on the image of $\sigma$ in the character variety $M_{k}(L, W, a)$. 

\begin{lemma}[Scaling invariance] \label{lem:ScalingInv}
Let $\sigma = \nabla + \sum_{i = 1}^{k-1} \eta_{i}$ be a splitting. There exists a positive $\epsilon(\sigma)$ such that the symplectic algebroids $(A(\sigma(s)), \omega(s))$ are isotopic for all $0 < s \leq \epsilon(\sigma)$. 
\end{lemma}
\begin{proof}
For $u \in \mathbb{R}_{>0}$, consider the map $m_{u} : L \to L$ defined by $m_{u}(x,t) = (x,ut)$. This lifts to an algebroid map $m_{u} : A(\sigma(s)) \to A(\sigma(u^{-1}s))$. Hence, define the following family of isotopies 
\[
\phi_{s}^{u} = m_{u^{-s}} : A(\sigma(s)) \to A(\sigma(u^s s)),
\] 
and note that it satisfies $\phi_{0}^{u} = id$. Furthermore, by using a bump function, we can modify this family so that it is supported on a compact tubular neighbourhood of $W$ in $L$. Therefore, by Corollary \ref{isotopyinv}, there is some positive number $a$ such that $(A(\sigma(s)), \omega(s)) \cong (A(\sigma(u^s s)), \omega(u^s s))$ for $0 \leq s \leq a$. In fact, since the proof of Corollary \ref{isotopyinv} can be done in families, the above isomorphisms also exist for $b \leq u \leq 1$ for some $0 < b < 1$. Therefore, given $0 < s \leq a$, the symplectic algebroids are isotopic to each other in the closed interval $[b^s s, s]$. By induction, the algebroids are isotopic in the closed interval $[s_{k}, a]$, where $s_{k}$ is the sequence defined recursively by $s_{1} = a$ and $s_{k+1} = b^{s_{k}} s_{k}$. Therefore, since $\lim_{k \to \infty} s_{k} = 0$, the result follows with $\epsilon(\sigma) = a$. 
\end{proof}

\begin{proposition}[Gauge invariance] \label{prop:GaugeInv}
Let $\sigma^1$ and $\sigma^2$ be splittings with underlying flat connection given by $\nabla$, and let $(A(\sigma^{i}(s)), \omega^i(s))$, for $i = 1, 2$, be the corresponding deformations of $(A(a), \omega)$. If $\sigma^1$ and $\sigma^2$ determine the same element of $M_{k}(L, W, a)$, then $(A(\sigma^{1}(s)), \omega^1(s))$ and $(A(\sigma^{2}(s)), \omega^2(s))$ are isotopic for small values of $s$. 
\end{proposition}
\begin{proof}
Since the splittings $\sigma^1$ and $\sigma^2$ determine the same element of the character variety, they are gauge equivalent via a gauge transformation $g: W \to G_{k}^{0} \cong K_{k} \ltimes \mathbb{R}_{>0}$. Using the factorization, and the fact that $K_{k}$ is unipotent and thus has a global logarithm, we can write $g = \exp(W) \exp(V)$, where $V = h(x) t \partial_t$ and $W = \sum_{i = 1}^{k-1} f_{i}(x) t^{i+1}\partial_{t}$. Let $m_{s} : L \to L$ be the rescaling map $m_{s}(x,t) = (x,st)$. Let $g_{s} = m_{s^{-1}} g m_{s}$, for $s > 0$. Then 
\[
g_{s} : A(\sigma^1(s)) \to A(\sigma^2(s)).
\]
But $g_{s} = \exp(W_{s}) \exp(V)$, where $W_s = \sum_{i = 1}^{k-1} s^i f_{i}(x) t^{i+1}\partial_{t}$. Sending $s \to 0$, we get $g_{0} = \exp(V) \in \mathrm{Aut}(A(a))$. But $\exp(V)_{\ast} : A(a) \to A(a + dh)$. Hence $dh = 0$ so that $h = c$ is a constant and $\exp(V) = m_{u^{-1}}$ for $u = e^{-c}$. Let $\phi_{s} = \exp(W_{s})$. Then 
\[
 \phi_{s} : A(\sigma^1(us)) \to A(\sigma^2(s))
\]
is a family of isotopies with $\phi_{0} = id$. By using a bump function, we can modify this family so that it is supported on a compact tubular neighbourhood of $W$ in $L$. By Corollary \ref{isotopyinv}, the symplectic algebroids $(A(\sigma^{1}(us)), \omega^1(us))$ and $(A(\sigma^{2}(s)), \omega^2(s))$ are isotopic for small $s$. Furthermore, by Lemma \ref{lem:ScalingInv}, $(A(\sigma^{1}(s)), \omega^1(s))$ and $(A(\sigma^{1}(us)), \omega^1(us))$ are also isotopic for small $s$. 
\end{proof}

\subsubsection{Main result}
Summarizing the results of this section we obtain the following. 
\begin{theorem} \label{th:cosympdef}
Let $(W, B, a)$ be a compact cosymplectic manifold of dimension $2n + 1$. Then there is a linear order $k + 1$ hypersurface algebroid $A(a)$ for $(W, L = W \times \mathbb{R})$ which is equipped with the following algebroid symplectic form 
\begin{equation} \label{eq:normalform}
\omega = B - ka \wedge \frac{dt}{t^{k+1}} \in \Omega^2_{A(a)}. 
\end{equation}
Let $M$ be a manifold containing $W$ as a hypersurface and whose normal bundle $\nu_{W}$ is trivializable. Then $M$ inherits an order $k+1$ HS algebroid for $(W,M)$, also denoted $A(a)$. Assume that $\omega$ extends to an algebroid symplectic form on all of $M$. Then, there is a map 
\[
\Phi : M_{k}(M, W, a) \to \mathrm{Pois}(M, W, \mathcal{F}),
\]
where 
\begin{itemize}
\item $M_{k}(M, W, a) = \mathrm{Hom}_{(L, [a])}(\pi_{1}(W), G_{k})/G_{k}^{0}$ is the $G_{k}$-character variety of $(W, M, a)$,  
\item $\mathrm{Pois}(M, W, \mathcal{F})$ is the space of Poisson structures on $M$ whose symplectic leaves consist of the components of $M \setminus W$ and the leaves of the foliation $\mathcal{F} = \ker(a)$, modulo isotopies preserving $W$ and $\mathcal{F}$. 
\end{itemize}
Given an element $\sigma \in M_{k}(M, W, a)$, with extension class $e(\sigma) \in H^2(W, L^{-k})$, the symplectic variation of $\Phi(\sigma)$ along $\mathcal{F}$ is given by the restriction $-e(\sigma)|_{\mathcal{F}} \in H^{2}_{\mathcal{F}}(\nu_{\mathcal{F}}^*)$, and is well-defined up to multiplication by positive functions which are constant along the leaves of $\mathcal{F}$. 
\end{theorem}

In Theorem \ref{th:cosympdef}, we consider HS symplectic structures on a manifold $M$ which extend the normal form of Equation \ref{eq:normalform}. The following result, which is folklore in the literature, provides a method of producing such structures from the data of a symplectic manifold $(M, \omega)$ with a cosymplectic submanifold $W$. Recall that $W$ is cosymplectic if there exists a vector field $X \in \mathfrak{X}(U)$ in a neighbourhood $U$ of $W$ which is both transverse to $W$ and is symplectic, in the sense that $L_{X}\omega = 0$. 

\begin{lemma} \label{lem:desing}
Let $(M, \omega)$ be a symplectic manifold and let $W \subset M$ be a compact cosymplectic submanifold. Let $k$ be an odd positive integer. Then there exists a $b^{k+1}$-symplectic form $\tilde{\omega}$ which agrees with $\omega$ outside of a tubular neighbourhood of $W$ and has local model around $W$ given by Equation \ref{eq:normalform}. 
\end{lemma}
\begin{remark}
This construction is known as the `singularization' of a symplectic form. The converse, namely the desingularization of a $b^k$-symplectic form, is described in \cite{MR3952555}. It is also possible to construct a singularization in the case that $k$ is even, but it requires doubling the singular hypersurface $W$. 
\end{remark}
\subsection{Mapping Tori} \label{sec:mapping} A compact cosymplectic manifold $(W, B, a)$ must be a mapping torus. Indeed, $a$ is a closed nowhere vanishing $1$-form and so the result follows by Tischler's theorem \cite{MR256413}. In this section, we specialize to the case where $a$ is pulled back from the base $S^1$ of the mapping torus and give a construction of order $4$ symplectic HS algebroids. 

Let $(N, B)$ be a compact symplectic manifold and let $\phi \in \mathrm{Symp}(N,B)$ be a symplectomorphism. Define the mapping torus 
\[
T(\phi) = (N \times \mathbb{R}) /\mathbb{Z}, 
\]
where $(\phi(n), \theta) \sim (n, \theta + 1)$. Then $(T(\phi), B, a = \lambda d \theta)$ is compact cosymplectic for every non-zero value of $\lambda$. Consider the trivial line bundle $L = T(\phi) \times \mathbb{R}$ with connection $\nabla = d - a$. Let $A(a)$ be the order $4$ HS algebroid for $(T(\phi),L)$ as in Section \ref{sec:cosymplectic}, and let $\omega$ be the algebroid symplectic form given by Equation \ref{eq:normalform}, where $k = 3$. 

First, we extend this HS symplectic structure to $M = T(\phi) \times S^1$, where $S^1 = \mathbb{R}/2 \pi \mathbb{Z}$ with coordinate $u$. The extension is given by 
\[
\omega = B - 3 a \wedge \frac{du}{16 \sin^4(\frac{u}{2})},
\]
which is locally equivalent to Equation \ref{eq:normalform} via a coordinate transformation.

The cohomology of $\nabla^{-i}$ is computed in Theorem \ref{lambdacohomology}. In particular, we find that 
\[
H^1_{T(\phi)}(L^{-1}) \cong H^{1}(N)_{e^{\lambda}}, \qquad H^1_{T(\phi)}(L^{-2}) \cong H^{1}(N)_{e^{2\lambda}},
\]
where $H^1(N)_{u}$ denotes the $u$-eigenspace of the map $\phi_{*} : H^1(N) \to H^1(N)$. By \cite[Proposition 10.9]{foliationpaper}, the $G_{3}$-character variety of $(T(\phi), M,a)$ is given by 
\[
M_{3}(M, T(\phi), a) \cong ( H^{1}(N)_{e^{\lambda}} \times H^{1}(N)_{e^{2\lambda}})/\mathbb{R}_{>0},
\]
where $\mathbb{R}_{>0}$ acts on the two factors with respective weights $-1$ and $-2$. Let $d(l) = \mathrm{dim}(H^1(N)_{e^l})$, then removing the origin from $M_{3}(M, W, a)$, we obtain a space which is homeomorphic to the sphere $S^{d(\lambda) + d(2 \lambda) - 1}$.

\begin{corollary} \label{rusapplication}
Let $(N,B)$ be a compact symplectic manifold with symplectomorphism $\phi \in \mathrm{Symp}(N,B)$, let $T(\phi)$ be the mapping torus, and let $\mathcal{F}$ be the foliation given by the fibres of the projection $\pi : T(\phi) \to S^1$. Let $H^1(N)_{u}$ denote the $u$-eigenspace of $\phi_{*} : H^1(N) \to H^1(N)$. Then, for every $\lambda \neq 0$, there is a map 
\[
\Phi: ( H^{1}(N)_{e^{\lambda}} \times H^{1}(N)_{e^{2\lambda}})/\mathbb{R}_{>0} \to \mathrm{Pois}(M, T(\phi), \mathcal{F}),
\]
where $\mathrm{Pois}(M, T(\phi), \mathcal{F})$ is the space of Poisson structures on $M = T(\phi) \times S^1$ whose symplectic leaves consist of $T(\phi) \times (S^1 \setminus \{ 1 \})$ and the leaves of $\mathcal{F}$, modulo isotopies preserving $T(\phi) \times \{ 1 \}$ and $\mathcal{F}$. 

Given an element $(\eta_1, \eta_2) \in H^{1}(N)_{e^{\lambda}} \times H^{1}(N)_{e^{2\lambda}}$, the symplectic variation of $\Phi(\eta_1, \eta_2)$ along $\mathcal{F}$, restricted to a fibre of $\pi$, is given by 
\[
\eta_2 \wedge \eta_1 \in H^1(N)_{e^{3 \lambda}},
\]
and is well-defined up to multiplication by a positive number. 
\end{corollary}

The following generalizes the example from Sections \ref{cat} and \ref{cat2} and produces arbitrarily large families of order 4 HS symplectic structures on compact $6$-dimensional manifolds. 

\begin{example}[Arnold's cat map revisited] \label{ex:lastcat} Let $S_{g}$ be a genus $g \geq 1$ surface and choose a symplectic form $\omega \in \Omega^2_{S_{g}}$. Choosing a symplectic basis of curves on $S_{g}$ fixes an isomorphism $H_1(S_{g}, \mathbb{Z}) \cong \mathbb{Z}^{2g}$ for which the intersection pairing is written in the standard form. Let $\Gamma = \begin{pmatrix} 2 & 1 \\ 1 & 1 \end{pmatrix} \in \mathrm{SL}(2, \mathbb{Z})$. Then $A = (\Gamma^{-1})^{\oplus g} \in \mathrm{Sp}(2g, \mathbb{Z})$ is a symplectic automorphism. By \cite[Theorem 6.4]{farb2011primer} there is a symplectomorphism $\psi \in \mathrm{Symp}(S_{g}, \omega)$ which realizes $A$, in the sense that the map induced by $\psi$ on $H_1(S_{g}, \mathbb{Z})$ is given by $A$. It follows that on cohomology, we have 
\[
\psi_{*} = A^{-1} = \Gamma^{\oplus g} : \mathbb{R}^{2g} \to \mathbb{R}^{2g},
\]
where we are using the symplectic basis of curves to identify $H^{1}_{S_{g}} \cong \mathbb{R}^{2g}$. The eigenvalues of $\psi_{*}$ are $r^{\pm1} = 1 \pm \varphi^{\pm1}$, where $\varphi = \frac{1 + \sqrt{5}}{2}$ is the golden ratio, and $\psi_*$ induces the eigenspace decomposition 
\[
H^1_{S_{g}} = E_{r} \oplus E_{r^{-1}}. 
\]
Note that $\mathrm{dim}(E_{r}) = \mathrm{dim}(E_{r^{-1}}) = g$. 

Now choose a pair of symplectic surfaces $(S_{g_{1}}, \omega_1), (S_{g_{2}}, \omega_2)$ with respective genera $g_{1}, g_{2} \geq 1$. Let $\psi_1 : S_{g_{1}} \to S_{g_{1}}$ be the symplectomorphism realizing $A$ and let $\psi_2 : S_{g_{2}} \to S_{g_{2}}$ be the symplectomorphism realizing $A^2$. Define $N = S_{g_{1}} \times S_{g_{2}}$, with symplectic form $B = \omega_1 + \omega_2$. The product $\phi = \psi_1 \times \psi_2 \in \mathrm{Symp}(N,B)$ is a symplectomorphism. It follows from the discussion above that the pushforward $\phi_*$ on cohomology induces the following eigenspace decomposition 
\[
H^1_{N} \cong H^{1}_{S_{g_{1}}} \oplus H^{1}_{S_{g_{2}}} = E_{r} \oplus E_{r^{-1}} \oplus E_{r^2} \oplus E_{r^{-2}},
\]
where the respective dimensions are $g_{1}, g_{1}, g_{2}, g_{2}$. Now let $M = T(\phi) \times S^1$, where $T(\phi)$ is the mapping torus equipped with the cosymplectic structure $(B, a = \log(r) d \theta)$. Then by Corollary \ref{rusapplication}, there is a map 
\[
\Phi : (E_{r} \times E_{r^2})/\mathbb{R}_{>0} \cong S^{g_1 + g_2 -1} \cup \{ 0 \} \to  \mathrm{Pois}(M, T(\phi), \mathcal{F}).
\]
Given a pair $(\eta_{1}, \eta_{2}) \in (E_{r} \times E_{r^2})/\mathbb{R}_{>0} $, the restriction of the symplectic variation of $\Phi(\eta_1, \eta_2)$ to a fibre of $\pi$ is given by 
\[
- \eta_1 \otimes \eta_2 \in E_{r} \otimes E_{r^2} / \mathbb{R}_{>0}. \hfill \qedhere
\]
\end{example}

\section{Open questions} \label{sec:ques}
We end the paper by raising a few remaining unanswered questions. First, as already noted, by Tischler's theorem \cite{MR256413}, the submanifolds $W$ in Theorem \ref{th:cosympdef} will always be mapping tori because a cosymplectic structure includes a nowhere vanishing closed $1$-form $a$. In terms of the data describing a closed $2$-form in Lemma \ref{lem:sympdata}, this is the form $\alpha_k$. In fact, this observation provides a general topological constraint on the degenerating hypersurface of any $b$ or $b^m$-symplectic manifold. However, for general HS symplectic structures, there is a possible loophole. Indeed, a closer look at the conditions in Lemma \ref{lem:sympdata} reveals that $\alpha_k$ need only be closed with respect to a \emph{flat connection} on the normal bundle $\nu_W$. Therefore, this leaves open the possibility of constructing HS symplectic forms where the compact hypersurface $W$ is not a mapping torus. We leave this as an open problem: 

\begin{question} 
Do there exist compact hypersurface symplectic algebroids $(A, \omega)$ for which the hypersurface $W$ is not a mapping torus? 
\end{question}

Finally, we note the following natural question, which does not appear to have been addressed in the literature (although, see \cite{Rankin26} for an analogous question studied in the setting of linear Poisson structures). 

\begin{question} \label{quest:lift}
Let $Q$ be a Poisson structure on a manifold $M$ which is generically symplectic but drops rank along a hypersurface $W$. What are the conditions on $Q$ so that it lifts to a symplectic form on a hypersurface algebroid? 
\end{question}

\begin{remark}
A Poisson structure $Q$ can lift to a symplectic form on several different hypersurface algebroids, a fact which is exploited in the proof of Theorem \ref{th:cosympdef}. In addition to Question \ref{quest:lift}, it also seems natural to ask for a classification of all hypersurface algebroids to which a given Poisson structure lifts. 
\end{remark}

\appendix
\section{Characteristic classes of (symplectic) foliations} \label{AppFol} 
In this appendix we recall some of the theory of codimension $1$ (symplectic) foliations and their characteristic classes. Let $D \subset TM$ be a codimension $1$ distribution. Let $\nu = TM/D$ be the normal bundle and let $\theta \in \Omega^{1}_{M}(\nu)$ denote the quotient map arising from the sequence 
\[
0 \to \mathcal{F} \to TM \to \nu \to 0. 
\]
Hence $D = \mathrm{ker}(\theta)$. The involutivity condition for $D$ can be stated purely in terms of $\theta$ as follows. 

\begin{lemma} \label{involutivitycondition}
Let $\nu$ be the normal bundle of a corank $1$ distribution $D$ on $M$ and let $\theta \in \Omega^1_{M}(\nu)$ be the associated $1$-form. Let $\nabla$ be a flat connection on $\nu$. Then $D =  \mathrm{ker}(\theta)$ is involutive if and only if 
\[
d^{\nabla} \theta = \beta \wedge \theta,
\]
for some $\beta \in \Omega^1_{M}$. 
\end{lemma}
Now let $\mathcal{F} \subset TM$ be a codimension $1$ foliation (i.e. an involutive distribution) and let $\iota^* : \Omega^{\bullet}_{M} \to \Omega^{\bullet}_{\mathcal{F}}$ be the restriction map. The kernel of $\iota^*$ is isomorphic to $\Omega^{\bullet - 1}_{\mathcal{F}} \otimes \nu^*$ and the differential on this subcomplex is induced by the natural Bott connection $\nabla^{Bott}$ on $\nu$, which is defined by bracketing sections of $\mathcal{F}$ and $\nu$. Hence, we obtain a short exact sequence of cochain complexes: 
\begin{equation} \label{eq:folseq}
0 \to \Omega^{\bullet - 1}_{\mathcal{F}} \otimes \nu^* \to  \Omega^{\bullet}_{M} \to \Omega^{\bullet}_{\mathcal{F}} \to 0.
\end{equation}
This induces a long exact sequence in cohomology, with connecting homomorphism 
\[
\delta_\nu: H^{k}_{\mathcal{F}} \to H^{k}_{\mathcal{F}}(\nu^*),
\]
which we name the \emph{variation map}. 

Given the foliation $\mathcal{F}$, choose a flat connection $\nabla$ on the normal bundle $\nu$. Then $d^{\nabla} \theta = \beta \wedge \theta$, for some $\beta \in \Omega^{1}_{M}$, as in Lemma \ref{involutivitycondition}. The restriction $\beta|_{\mathcal{F}} \in \Omega_{\mathcal{F}}^{1}$ is a well-defined closed foliated $1$-form. The cohomology class 
\[
m(\mathcal{F}, \nabla) = [\beta|_{\mathcal{F}}] \in H^{1}(\mathcal{F})
\]
is the \emph{modular class} of $(\mathcal{F}, \nabla)$. Note that this class does depend on the choice of connection. Indeed, if $\nabla' = \nabla + \alpha$ is another flat connection, for $\alpha \in \Omega^1_{M}$ a closed $1$-form, then the modular class changes by adding the cohomology class of this form:
\[
m(\mathcal{F}, \nabla') = m(\mathcal{F}, \nabla) + \iota^*[\alpha]. 
\]
However, since gauge transformations only affect $\nabla$ by adding \emph{exact} forms, they do not change the modular class. 

\begin{lemma}
The modular class $(\mathcal{F}, \nabla)$ only depends on $\nabla$ through its gauge equivalence class. 
\end{lemma}
On the other hand, since $ \iota^*[\alpha]$ is in the kernel of the connecting homomorphism $\delta_{\nu}$, the variation of $m(\mathcal{F}, \nabla)$ is independent of $\nabla$. 

\begin{definition} \label{def:varfol}
The \emph{variation} of a corank $1$ foliation $\mathcal{F}$ is defined to be the cohomology class
\[
\mathrm{var}(\mathcal{F}) = \delta_{\nu}(m(\mathcal{F}, \nabla)) \in H^{1}_{\mathcal{F}}(\nu^*).
\]
We say that $\mathcal{F}$ has \emph{trivial variation} if $\mathrm{var}(\mathcal{F}) = 0$. 
\end{definition}

\begin{proposition} \label{prop:bottextend}
Consider a corank $1$ foliation $\mathcal{F}$, equipped with a flat connection $\nabla$ on the normal bundle $\nu$, let $\theta \in \Omega^1_{M}(\nu)$ be the associated $1$-form, and let $m = m(\mathcal{F}, \nabla)  \in H^{1}(\mathcal{F})$ be the modular class. Then the following statements are equivalent 
\begin{itemize}
\item $\mathrm{var}(\mathcal{F}) = \delta_{\nu}(m) = 0$, 
\item there exists a flat connection $\tilde{\nabla}$ on $\nu$ such that $d^{\tilde{\nabla}}\theta = 0$,
\item there exists a flat extension of the Bott connection on $\nu$. 
\end{itemize}
Furthermore, if these conditions are satisfied, then the flat extension of the Bott connection is unique up to an element of $H_{\mathcal{F}}^{0}(\nu^*)$. In other words, it is unique up to a flat section of the Bott connection on $\nu^*$. 
\end{proposition}
Finally, recall that a \emph{symplectic foliation} consists of the data $(\mathcal{F}, \omega_{\mathcal{F}})$ of a foliation $\mathcal{F}$ along with a closed foliated $2$-form $\omega_{\mathcal{F}} \in \Omega^2_{\mathcal{F}}$ which restricts to a symplectic form on each leaf. 

\begin{definition}
The \emph{symplectic variation} of a symplectic foliation $(\mathcal{F}, \omega_{\mathcal{F}})$ is defined to be the cohomology class 
\[
\mathrm{var}(\omega_{\mathcal{F}}) = \delta_{\nu}[\omega_{\mathcal{F}}] \in H^{2}_{\mathcal{F}}(\nu^*). \hfill\qedhere
\]
\end{definition}

\section{Cohomology of mapping tori} \label{app:cohomappingtori} Let $M$ be a compact manifold and let $\alpha$ be a closed nowhere vanishing $1$-form. By Tischler's theorem \cite{MR256413}, $M$ is a fibre bundle over $S^1$ and $\alpha$ may be approximated by the pullback of $d\theta \in \Omega^1(S^1)$. Hence, let's assume that we have a surjective submersion $f: M \to S^1$ and that $\alpha = f^*(d\theta)$.  

Next, let $H$ be a lift of the vector field $\partial_{\theta}$, and let $\phi: M \to M$ be the time $1$ flow, which covers the identity map on the base $S^1$. Let $N = f^{-1}(1)$ be the fibre above $1$. Then $M$ is the mapping torus for $\phi$. Namely, $M \cong (N \times \mathbb{R})/\mathbb{Z}$, where the action of $\mathbb{Z}$ is given by 
\[
k \ast (n, \theta) = (\phi^{-k}(n), \theta + k). 
\]
The aim of this Appendix is to study the cohomology $H^{\bullet}_{M}(\lambda)$ of the twisted de Rham complex 
\[
\Omega^{\bullet}_{\lambda} = (\Omega^{\bullet}_{M}, d_{\lambda} = d + \lambda \alpha \wedge ).
\] 
Let $\mathcal{F} = \mathrm{ker}(\alpha)$ be the foliation given by the fibres of $f$. The twisted de Rham complex sits in an exact sequence with the foliated de Rham complex, generalizing Equation \ref{eq:folseq} : 
\begin{equation} \label{eq:twistedfol}
0 \to \Omega_{\mathcal{F}}^{\bullet - 1} \to \Omega^{\bullet}_{\lambda} \to \Omega^{\bullet}_{\mathcal{F}} \to 0,
\end{equation}
where the left hand map is given by wedging $\alpha \wedge \ $. Using the vector field $H$, which satisfies $\alpha(H) = 1$, we can split this sequence to obtain 
\[
\Omega_{M}^{r} = \Omega_{\mathcal{F}}^{r} \oplus \alpha \wedge \Omega_{\mathcal{F}}^{r-1}. 
\]
Using this decomposition, the twisted de Rham differential has the form 
\[
d_{\lambda}(\beta + \alpha \wedge \eta) = d_{\mathcal{F}}\beta + \alpha\wedge(( L_{H}\beta + \lambda \beta) - d_{\mathcal{F}}\eta). 
\]
From this, we immediately obtain the following result. 
\begin{lemma}
The short exact sequence \ref{eq:twistedfol} is compatible with differentials, where the left complex is equipped with $-d_{\mathcal{F}}$, the central complex is equipped with $d_{\lambda}$, and the right complex is equipped with $d_{\mathcal{F}}$. Furthermore, the connecting homomorphism of the induced long exact sequence is given by  \[ L_{H} + \lambda id: H^{i}_{\mathcal{F}} \to H^{i}_{\mathcal{F}}. \] 
\end{lemma}

Now we take a closer look at the foliated de Rham complex $\Omega^{\bullet}_{\mathcal{F}}$. Note first that using the pullback by $f$, this vector space is a module over $C^{\infty}(S^1)$. Furthermore, the foliated differential $d_{\mathcal{F}}$ is $C^{\infty}(S^1)$-linear, implying that $H^{i}_{\mathcal{F}}$ is also a $C^{\infty}(S^1)$-module. In fact, this module comes from a vector bundle $\mathcal{H}^{i}$ over the circle: 
\[
H^{i}_{\mathcal{F}} = \Gamma(S^1, \mathcal{H}^{i}). 
\]
The fibres of this vector bundle are 
\[
\mathcal{H}^{i}|_{\theta} = H^{i}(f^{-1}(\theta)). 
\]
Furthermore, the lattice $H^{i}(f^{-1}(\theta), \mathbb{Z})^{\mathrm{free}} \subset H^{i}(f^{-1}(\theta))$ defines the flat Gauss-Manin connection $\nabla^{GM}$ on the vector bundle. 

\begin{lemma}
The operator $L_{H}$ descends to $\nabla^{GM}_{\partial_{\theta}}$ on cohomology. 
\end{lemma}

We may therefore write the long exact sequence induced by the short exact sequence \ref{eq:twistedfol} as follows: 
\begin{equation}\label{longexact}
\cdots \to \Gamma(\mathcal{H}^{i-1}) \to H^{i}_{M}(\lambda) \to \Gamma(\mathcal{H}^{i}) \to  \Gamma(\mathcal{H}^{i}) \to H^{i+1}_{M}(\lambda) \to \cdots
\end{equation}
where the connecting homomorphism is given by 
\[
\nabla^{GM}_{\partial_{\theta}} + \lambda id : \Gamma(\mathcal{H}^{i}) \to  \Gamma(\mathcal{H}^{i}). 
\]
This gives us the following description of $d_{\lambda}$-cohomology: 
\begin{equation} \label{ses2}
0 \to H^{1}_{\lambda}(S^{1}, \mathcal{H}^{r-1}) \to H^{r}_{M}(\lambda) \to H^{0}_{\lambda}(S^1, \mathcal{H}^{r}) \to 0,
\end{equation}
where the right and left hand groups are the cohomologies of the connection $\nabla^{GM} + \lambda d\theta$: 
\[
H^{0}_{\lambda}(S^1, \mathcal{H}^{r}) = \ker(\nabla^{GM}_{\partial_{\theta}} + \lambda id ), \qquad H^{1}_{\lambda}(S^{1}, \mathcal{H}^{r-1})  = \mathrm{coker}(\nabla^{GM}_{\partial_{\theta}} + \lambda id ). 
\]
These can be computed using the following fact about the cohomology of local systems over the circle. 

\begin{lemma} \label{localsystemcoh}
Let $(E, \nabla)$ be a vector bundle with connection over the circle. Let $M : E_{1} \to E_{1}$ be the monodromy automorphism, obtained by taking the parallel transport around the circle. Then 
\[
H^{0}(E, \nabla) \cong \mathrm{ker}(M - id), \qquad H^{1}(E, \nabla) \cong \mathrm{coker}(M - id). 
\]
\end{lemma}

The following result computes the twisted cohomology groups. 
\begin{theorem} \label{lambdacohomology}
Let $M$ be the mapping torus for $\phi : N \to N$, and let $\alpha = f^*(d\theta)$. Then the cohomology groups $H^{\bullet}_{M}(\lambda)$ of the twisted de Rham complex $(\Omega^{\bullet}_{M}, d_{\lambda} = d + \lambda \alpha \wedge )$, for $\lambda \in \mathbb{R}$, are described by the following short exact sequences 
\[
0 \to H^{r-1}(N)_{e^{-\lambda}\phi_{*}} \to H^{r}_{M}(\lambda) \to H^{r}(N)^{e^{-\lambda}\phi_{*}}  \to 0, 
\]
where $H^{r}(N)^{e^{-\lambda}\phi_{*}} $ is the $e^{\lambda}$ eigenspace of $\phi_{*} : H^{r}(N) \to H^{r}(N)$, and $H^{r-1}(N)_{e^{-\lambda}\phi_{*}} $ is the cokernel of 
\[
 \phi_{*} - e^{\lambda} : H^{r-1}(N) \to H^{r-1}(N). 
 \]
\end{theorem}
\begin{proof}
Starting from the short exact sequence \ref{ses2} and Lemma \ref{localsystemcoh}, it suffices to determine the monodromy of $\nabla^{GM} + \lambda d\theta$. First, the monodromy of the Gauss-Manin connection $\nabla^{GM}$ is given by the pushforward map $\phi_{*}: H^{r}(N) \to H^{r}(N)$. Given a flat section $s$ with respect to $\nabla^{GM}$, we compute 
\[
\nabla^{GM}_{\partial_{\theta}}(e^{-\lambda \theta} s) = \partial_{\theta}(e^{-\lambda \theta}) s = -\lambda e^{-\lambda \theta} s,
\]
implying that $e^{-\lambda \theta} s$ is flat with respect to $\nabla^{GM} + \lambda d\theta$. It follows that the monodromy of $\nabla^{GM} + \lambda d\theta$ is given by $e^{-\lambda} \phi_{*}$. The result then follows from the observation that $e^{-\lambda} \phi_{*} - id$ has the same kernel and cokernel as $\phi_{*} - e^{\lambda} id$. 
\end{proof}

\bibliographystyle{alpha}
\bibliography{references} 

@misc{Rankin26,
      title={Algebroid Desingularizable Poisson Structures}, 
      author={Shane Rankin},
      year={2026},
      eprint={2605.22519},
      archivePrefix={arXiv},
      primaryClass={math.DG},
      url={https://arxiv.org/abs/2605.22519}, 
}

@article{NestTsygan2001,
  author    = {Nest, Ryszard and Tsygan, Boris},
  title     = {Deformations of symplectic {L}ie algebroids, deformations of holomorphic symplectic structures, and index theorems},
  journal   = {Asian Journal of Mathematics},
  volume    = {5},
  number    = {4},
  pages     = {599--635},
  year      = {2001},
  doi       = {10.4310/AJM.2001.v5.n4.a2},
  publisher = {International Press of Boston}
}

@article{francis2024singular,
  title={On singular foliations tangent to a given hypersurface},
  author={Francis, Michael D},
  journal={Journal of Noncommutative Geometry},
  year={2024}
}

@article{matviichuk2020local,
  title={A local Torelli theorem for log symplectic manifolds},
  author={Matviichuk, Mykola and Pym, Brent and Schedler, Travis},
  journal={arXiv preprint arXiv:2010.08692},
  year={2020}
}

@book{farb2011primer,
  title={A primer on mapping class groups},
  author={Farb, Benson and Margalit, Dan},
  volume={49},
  year={2011},
  publisher={Princeton university press}
}

@article{babenko1998nonformal,
  title={On nonformal simply-connected symplectic manifolds},
  author={Babenko, Ivan K and Taimanov, Iskander A},
  journal={Siberian Mathematical Journal},
  volume={41},
  number={2},
  pages={204--217},
  year={2000},
}

@article{nomizu1954cohomology,
  title={On the cohomology of compact homogeneous spaces of nilpotent Lie groups},
  author={Nomizu, Katsumi},
  journal={Annals of Mathematics},
  volume={59},
  number={3},
  pages={531--538},
  year={1954},
  publisher={JSTOR}
}

@article{babenko2000massey,
  title={Massey products in symplectic manifolds},
  author={Babenko, Ivan K and Taimanov, Iskander A},
  journal={Sbornik: Mathematics},
  volume={191},
  number={8},
  pages={1107},
  year={2000},
  publisher={IOP Publishing}
}

@article{foliationpaper,
      title={$b^k$-algebroids and the variety of foliation jets}, 
      author={Francis Bischoff and Alvaro del Pino and Aldo Witte},
      year={2025},
      eprint={2508.20241},
      archivePrefix={arXiv},
      primaryClass={math.DG},
      url={https://arxiv.org/abs/2508.20241}, 
      journal={arXiv preprint arXiv:2508.20241},
}

@article{bischoff2023jets,
  title={Jets of foliations and bk-algebroids},
  author={Francis Bischoff and Alvaro del Pino and Aldo Witte},
  journal={arXiv preprint arXiv:2311.17045},
  year={2023}
}

@article {MR256413,
    AUTHOR = {Tischler, D.},
     TITLE = {On fibering certain foliated manifolds over {$S^{1}$}},
   JOURNAL = {Topology},
  FJOURNAL = {Topology. An International Journal of Mathematics},
    VOLUME = {9},
      YEAR = {1970},
     PAGES = {153--154},
      ISSN = {0040-9383},
   MRCLASS = {57.36},
  MRNUMBER = {256413},
MRREVIEWER = {H. Suzuki},
       DOI = {10.1016/0040-9383(70)90037-6},
       URL = {https://doi-org.libproxy.uregina.ca/10.1016/0040-9383(70)90037-6},
}

@article{deligne1971hodge,
  title={Hodge theory. II},
  author={Deligne, Pierre},
  journal={Publication Mathematique IHES},
  volume={40},
  pages={5--58},
  year={1971}
}

@article {MR3952555,
    AUTHOR = {Guillemin, Victor and Miranda, Eva and Weitsman, Jonathan},
     TITLE = {Desingularizing {$b^m$}-symplectic structures},
   JOURNAL = {Int. Math. Res. Not. IMRN},
  FJOURNAL = {International Mathematics Research Notices. IMRN},
    VOLUME = {2019},
      YEAR = {2019},
    NUMBER = {10},
     PAGES = {2981--2998},
      ISSN = {1073-7928,1687-0247},
   MRCLASS = {53D17 (53D05)},
  MRNUMBER = {3952555},
MRREVIEWER = {Mohamed\ Boucetta},
       DOI = {10.1093/imrn/rnx126},
       URL = {https://doi.org/10.1093/imrn/rnx126},
}

@article {MR4236806,
    AUTHOR = {Miranda, Eva and Scott, Geoffrey},
     TITLE = {The geometry of {$E$}-manifolds},
   JOURNAL = {Rev. Mat. Iberoam.},
  FJOURNAL = {Revista Matem\'{a}tica Iberoamericana},
    VOLUME = {37},
      YEAR = {2021},
    NUMBER = {3},
     PAGES = {1207--1224},
      ISSN = {0213-2230,2235-0616},
   MRCLASS = {53D17 (57R30)},
  MRNUMBER = {4236806},
MRREVIEWER = {Tomasz\ Rybicki},
       DOI = {10.4171/rmi/1232},
       URL = {https://doi.org/10.4171/rmi/1232},
}

@article {MR4257086,
    AUTHOR = {Guillemin, Victor W. and Miranda, Eva and Weitsman, Jonathan},
     TITLE = {On geometric quantization of {$b^m$}-symplectic manifolds},
   JOURNAL = {Math. Z.},
  FJOURNAL = {Mathematische Zeitschrift},
    VOLUME = {298},
      YEAR = {2021},
    NUMBER = {1-2},
     PAGES = {281--288},
      ISSN = {0025-5874,1432-1823},
   MRCLASS = {53D50 (37J06 53D17)},
  MRNUMBER = {4257086},
MRREVIEWER = {Igor\ Mencattini},
       DOI = {10.1007/s00209-020-02590-w},
       URL = {https://doi.org/10.1007/s00209-020-02590-w},
}

@article {MR4523256,
    AUTHOR = {Cardona, Robert and Miranda, Eva},
     TITLE = {Integrable systems on singular symplectic manifolds: from
              local to global},
   JOURNAL = {Int. Math. Res. Not. IMRN},
  FJOURNAL = {International Mathematics Research Notices. IMRN},
    VOLUME = {2022},
      YEAR = {2022},
    NUMBER = {24},
     PAGES = {19565--19616},
      ISSN = {1073-7928,1687-0247},
   MRCLASS = {37J06 (37J35 53D05)},
  MRNUMBER = {4523256},
       DOI = {10.1093/imrn/rnab253},
       URL = {https://doi.org/10.1093/imrn/rnab253},
}

@article{gualtTrop,
	Author = {Gualtieri, Marco and Li, Songhao and Pelayo, {\'A}lvaro and Ratiu, Tudor S.},
	Da = {2017/04/01},
	Doi = {10.1007/s00208-016-1427-9},
	Id = {Gualtieri2017},
	Isbn = {1432-1807},
	Journal = {Mathematische Annalen},
	Number = {3},
	Pages = {1217--1258},
	Title = {The tropical momentum map: a classification of toric log symplectic manifolds},
	Ty = {JOUR},
	Url = {https://doi.org/10.1007/s00208-016-1427-9},
	Volume = {367},
	Year = {2017}}

@book {MR0417174,
    AUTHOR = {Deligne, Pierre},
     TITLE = {\'{E}quations diff\'{e}rentielles \`a points singuliers
              r\'{e}guliers},
    SERIES = {Lecture Notes in Mathematics},
    VOLUME = {Vol. 163},
 PUBLISHER = {Springer-Verlag, Berlin-New York},
      YEAR = {1970},
     PAGES = {iii+133},
   MRCLASS = {14D05 (14C30)},
  MRNUMBER = {417174},
MRREVIEWER = {Helmut\ Hamm},
}

@article{Radko2001ACO,
  title={A classification of topologically stable Poisson structures on a compact oriented structure},
  author={Olga Radko},
  journal={Journal of Symplectic Geometry},
  year={2001},
  volume={1},
  pages={523-542},
  url={https://api.semanticscholar.org/CorpusID:122897797}
}

@article{MR3523250,
	Author = {Scott, Geoffrey},
	Doi = {10.4310/JSG.2016.v14.n1.a3},
	Fjournal = {The Journal of Symplectic Geometry},
	Issn = {1527-5256},
	Journal = {J. Symplectic Geom.},
	Mrclass = {53D17 (58J40)},
	Mrnumber = {3523250},
	Mrreviewer = {Alberto Parmeggiani},
	Number = {1},
	Pages = {71--95},
	Title = {The geometry of {$b^k$} manifolds},
	Url = {https://doi-org.ezproxy-prd.bodleian.ox.ac.uk/10.4310/JSG.2016.v14.n1.a3},
	Volume = {14},
	Year = {2016}}

@article{saito1980theory,
	Author = {Saito, Kyoji},
	Journal = {J. Fac. Sci. Univ. Tokyo Sect. IA Math},
	Number = {2},
	Pages = {265--291},
	Title = {Theory of logarithmic differential forms and logarithmic vector fields},
	Volume = {27},
	Year = {1980}}

@book {MR1348401,
    AUTHOR = {Melrose, Richard B.},
     TITLE = {The {A}tiyah-{P}atodi-{S}inger index theorem},
    SERIES = {Research Notes in Mathematics},
    VOLUME = {4},
 PUBLISHER = {A K Peters, Ltd., Wellesley, MA},
      YEAR = {1993},
     PAGES = {xiv+377},
      ISBN = {1-56881-002-4},
   MRCLASS = {58G10 (58G15 58G25)},
  MRNUMBER = {1348401},
MRREVIEWER = {Rafe Mazzeo},
       DOI = {10.1016/0377-0257(93)80040-i},
       URL = {https://doi-org.ezproxy-prd.bodleian.ox.ac.uk/10.1016/0377-0257(93)80040-i},
}

@article {MR4229238,
    AUTHOR = {Lanius, Melinda},
     TITLE = {Symplectic, {P}oisson, and contact geometry on scattering
              manifolds},
   JOURNAL = {Pacific J. Math.},
  FJOURNAL = {Pacific Journal of Mathematics},
    VOLUME = {310},
      YEAR = {2021},
    NUMBER = {1},
     PAGES = {213--256},
      ISSN = {0030-8730},
   MRCLASS = {53D17 (53D05 53D10)},
  MRNUMBER = {4229238},
MRREVIEWER = {Joseph Dongho},
       DOI = {10.2140/pjm.2021.310.213},
       URL = {https://doi-org.ezproxy-prd.bodleian.ox.ac.uk/10.2140/pjm.2021.310.213},
}

@article{Gualtieri-Li-2012,
	Author = {Gualtieri, M. and Li, S.},
	Journal = {International Mathematics Research Notices},
	Number = {11},
	Pages = {3022--3074},
	Publisher = {OUP},
	Title = {Symplectic groupoids of log symplectic manifolds},
	Volume = {2014},
	Year = {2014}}

@article{GUILLEMIN2014864,
	Author = {Victor Guillemin and Eva Miranda and Ana Rita Pires},
	Doi = {https://doi.org/10.1016/j.aim.2014.07.032},
	Issn = {0001-8708},
	Journal = {Advances in Mathematics},
	Pages = {864-896},
	Title = {Symplectic and Poisson geometry on b-manifolds},
	Url = {https://www.sciencedirect.com/science/article/pii/S0001870814002722},
	Volume = {264},
	Year = {2014}}

@article{BLM19,
	Author = {Bursztyn, Henrique and Lima, Hudson and Meinrenken, Eckhard},
	Doi = {10.1515/crelle-2017-0014},
	Fjournal = {Journal f\"{u}r die Reine und Angewandte Mathematik. [Crelle's Journal]},
	Issn = {0075-4102},
	Journal = {J. Reine Angew. Math.},
	Mrclass = {70G45 (53D05)},
	Mrnumber = {4000576},
	Mrreviewer = {Camelia Arie\c{s}anu},
	Pages = {281--312},
	Title = {Splitting theorems for {P}oisson and related structures},
	Url = {https://doi-org.ezproxy-prd.bodleian.ox.ac.uk/10.1515/crelle-2017-0014},
	Volume = {754},
	Year = {2019}}

@article{Wei00,
	Author = {Weinstein, Alan},
	Doi = {10.1007/s100970050014},
	Fjournal = {Journal of the European Mathematical Society (JEMS)},
	Issn = {1435-9855},
	Journal = {J. Eur. Math. Soc. (JEMS)},
	Mrclass = {53C40 (53C20 57S15)},
	Mrnumber = {1750452},
	Mrreviewer = {McKenzie Y. Wang},
	Number = {1},
	Pages = {53--86},
	Title = {Almost invariant submanifolds for compact group actions},
	Url = {https://doi-org.ezproxy-prd.bodleian.ox.ac.uk/10.1007/s100970050014},
	Volume = {2},
	Year = {2000}}

@article{Fer02,
	Author = {Fernandes, Rui Loja},
	Doi = {10.1006/aima.2001.2070},
	Fjournal = {Advances in Mathematics},
	Issn = {0001-8708},
	Journal = {Adv. Math.},
	Mrclass = {58H05 (53C05 53C12 53C29 53D17 57R20 57R30)},
	Mrnumber = {1929305},
	Mrreviewer = {Jan Kubarski},
	Number = {1},
	Pages = {119--179},
	Title = {Lie algebroids, holonomy and characteristic classes},
	Url = {https://doi-org.ezproxy-prd.bodleian.ox.ac.uk/10.1006/aima.2001.2070},
	Volume = {170},
	Year = {2002}}

@incollection{Duf01,
	Author = {Dufour, Jean-Paul},
	Booktitle = {Lie algebroids and related topics in differential geometry ({W}arsaw, 2000)},
	Doi = {10.4064/bc54-0-3},
	Mrclass = {53D17 (22A99 58H05)},
	Mrnumber = {1881647},
	Mrreviewer = {Rui Loja Fernandes},
	Pages = {35--41},
	Publisher = {Polish Acad. Sci. Inst. Math., Warsaw},
	Series = {Banach Center Publ.},
	Title = {Normal forms for {L}ie algebroids},
	Url = {https://doi-org.ezproxy-prd.bodleian.ox.ac.uk/10.4064/bc54-0-3},
	Volume = {54},
	Year = {2001}}
\end{document}